\documentclass[11pt,a4paper]{article}
\usepackage[utf8]{inputenc}
\usepackage{amsmath}
\usepackage{amsfonts}
\usepackage{amssymb}
\usepackage{amsthm}
\usepackage{makeidx}
\usepackage{graphicx}
\usepackage{authblk}
\usepackage{subfig}
\usepackage{subeqnarray}
\usepackage{comment}
\usepackage[T1]{fontenc}
\usepackage[top=2cm, bottom=2cm, left=2cm, right=2cm]{geometry}
\usepackage{marginnote}

\usepackage{imakeidx}
\usepackage{hyperref}
\usepackage{color}
\makeindex[intoc]

\usepackage{enumerate}

\newtheorem{theorem}{Theorem}[section]
\newtheorem{corollary}[theorem]{Corollary}

\newtheorem{definition}[theorem]{Definition}
\newtheorem{proposition}[theorem]{Proposition}
\newtheorem{remark}{Remark}[section]

\title{The $p\,$-Laplacian equation in a rough thin domain with terms concentrating on the boundary}
\date{}

\author[1]{Ariadne Nogueira\thanks{e-mail: ariadnen@ime.usp.br}}
\author[1]{Jean Carlos Nakasato\thanks{e-mail: nakasato@ime.usp.br}}
\affil[1]{Depto. Matem\'atica Aplicada, Instituto de Matem\'atica e Estat\'istica, Universidade de S\~ao Paulo,
	Rua do Mat\~ao 1010, S\~ao Paulo - SP, Brazil}

\begin{document}


	\maketitle
		\begin{abstract}
		In this work we use reiterated homogenization and unfolding operator approach to study the asymptotic behavior of the solutions of the $p$-Laplacian equation with Neumann boundary conditions set in a rough thin domain with concentrated terms on the boundary. 
		We study weak, resonant and high roughness, respectively. In the three cases, we deduce the effective equation capturing the dependence on the geometry of the thin channel and the neighborhood where the concentrations take place. 	
		\end{abstract}


	\section{Introduction}
	
	In this work, we are interested in analyzing the asymptotic behavior of solutions of a quasilinear elliptic problem posed in a family of thin domains $R^\varepsilon$ with forcing terms concentrated on a neighborhood $\mathcal{O}^\varepsilon \subset R^\varepsilon$ of the boundary $\partial R^\varepsilon$. We assume
	\begin{equation}\label{TDs}
	R^\varepsilon=\left\lbrace (x,y)\in \mathbb{R}^2:0<x<1, 0<y<\varepsilon g\left(\dfrac{x}{\varepsilon^\alpha}\right)\right\rbrace,\,\,\,\,0<\varepsilon\ll 1,
	\end{equation}
	for any $\alpha > 0$, where 
	
	\par\medskip 	
	
	($\mathbf{H_g}$) {\sl $g : \mathbb{R}\rightarrow \mathbb{R}$ is a strictly positive, bounded, Lipschitz, $L_g$-periodic function and differentiable almost everywhere.  Moreover, we define $$g_0 = \min_{x \in \mathbb{R}} g(x) \quad  \textrm{ and } \quad g_1 = \max_{x \in \mathbb{R}} g(x)$$ 
		so that $0<g_0\leq g(x)\leq g_1$ for all $x\in\mathbb{R}$.}
			
	On the other hand, we set the narrow strip $\mathcal{O}^\varepsilon$ by 
	\begin{equation*} \label{strip}
	\mathcal{O}^\varepsilon=\left\lbrace (x,y)\in \mathbb{R}^2:0<x<1, \varepsilon \left[g\left(\dfrac{x}{\varepsilon^\alpha}\right)-\varepsilon^\gamma h\left(\dfrac{x}{\varepsilon^\beta}\right)\right]<y<\varepsilon g\left(\dfrac{x}{\varepsilon^\alpha}\right)\right\rbrace
	\end{equation*}
	with $\gamma$, $\beta$, and $\alpha>0$, and $h:\mathbb{R}\to\mathbb{R}$ being a positive function of class $\mathcal{C}^1$, $L_h$-periodic with bounded derivatives. 
	
	Notice that parameters $\alpha$ and $\beta$ set respectively the roughness order of the upper boundary of the thin domain $R^\varepsilon$ and the singular shape of the $\varepsilon^\gamma$-neighborhood $\mathcal{O}^\varepsilon$, whereas their profile are given by positive and periodic functions $g$ and $h$. Finally, the parameter $\gamma>0$ only establishes the order of the Lesbegue measure of $\mathcal{O}^\varepsilon$ with respect to $R^\varepsilon$.

	We first analyze the solutions of the problem 
	\begin{equation}\label{variational_l}\int_{R^\varepsilon} \{|\nabla u_\varepsilon|^{p-2}\nabla u_\varepsilon\nabla \varphi+|u_\varepsilon|^{p-2}u_\varepsilon\varphi \}dxdy=\dfrac{1}{\varepsilon^\gamma}\int_{\mathcal{O}_\varepsilon}f^\varepsilon\varphi dxdy, \quad \varphi\in W^{1,p}(R^\varepsilon),\end{equation}
which is the variational formulation of the quasi-linear equation 
	\begin{equation}\label{problem_l}
	\left\lbrace\begin{array}{lll}
	-\Delta_p u_\varepsilon+|u_\varepsilon|^{p-2}u_\varepsilon=\dfrac{1}{\varepsilon^\gamma}\chi_{\mathcal{O}^\varepsilon}f^\varepsilon\mbox{ in }R^\varepsilon\\[0.3cm]
	|\nabla u_\varepsilon|^{p-2}\dfrac{\partial u_\varepsilon}{\partial \nu^\varepsilon}=0\mbox{ on }\partial R^\varepsilon 
	\end{array}\right. .
	\end{equation}
	Here $\nu^\varepsilon$ denotes the unit outward normal to the boundary $\partial R^\varepsilon$ and, for $1<p<\infty$, $\Delta_p\,\cdot\,$
	is the $p$-Laplacian differential operator. Consider $\chi_{\mathcal{O}^\varepsilon}$ the characteristic function of the set $\mathcal{O}^\varepsilon$
	and we take forcing terms $f^\varepsilon\in L^{p'}(R^\varepsilon)$ for $p'>0$, with $1/p'+1/p=1$.

Notice that $R^\varepsilon \subset (0,1) \times (0, \varepsilon g_1)$ for all $\varepsilon>0$ degenerating to the unit interval as $\varepsilon \to 0$.  
Hence, it is reasonable to expect that the family of solutions $u_\varepsilon$ will converge to a 
solution of a one-dimensional equation of the same type, with homogeneous Neumann boundary condition, capturing the effect of the thin domain $R^\varepsilon$ and narrow strip $\mathcal{O}^\varepsilon$. 

Following previous works as \cite{arrieta, AMA-thin, SMarcone, PazaninPereira}, we use the characteristic function $\chi_{\mathcal{O}^\varepsilon}$ and term $1/\varepsilon^\gamma$ to express concentration on $\mathcal{O}^\varepsilon \subset R^\varepsilon$.
We obtain three different limit problems according to the oscillatory order established by parameter $\alpha$: assuming $0<\alpha<1$, we get the effective equation for that we call weak oscillation case; setting $\alpha=1$, we obtain the known resonant case; and taking $\alpha>1$, we consider the high oscillating scene. See Figure  \ref{fig1}, where the three cases are illustrated. 
	
	\begin{figure}[!htb]
		\centering
		\subfloat[Weak oscillation: $\alpha<1$.]{
			\scalebox{0.35}{\includegraphics{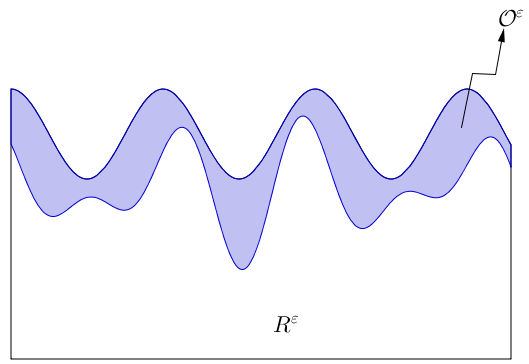}}
			\label{figdroopy}
		}
		\quad 
		\subfloat[Resonant case: $\alpha=1$.]{
			\scalebox{0.35}{\includegraphics{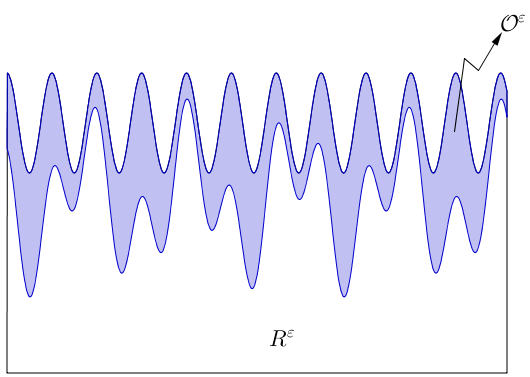}}
			\label{figsnoop}
		}
		\quad
		\subfloat[High oscillation: $\alpha>1$.]{
			\scalebox{0.35}{\includegraphics{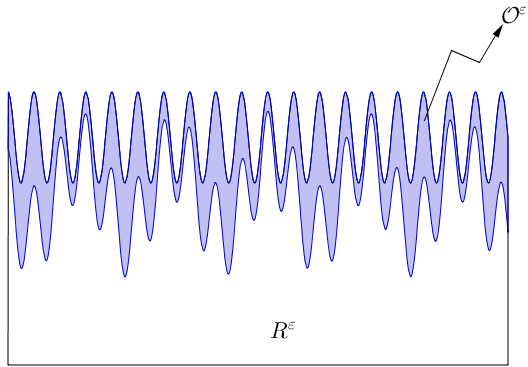}}
			\label{figdroopy1}
		}
		\caption{Examples of oscillatory boundaries set by $\alpha>0$.}
		\label{fig1}
	\end{figure}\newpage 
	
		In order to accomplish our goal, we combine and adapt techniques developed to deal with asymptotic analysis for singular boundary value problems. We use the unfolding operator method for thin domains from \cite{AM,AM2}, some monotone techniques used for the treatment of the $p$-Laplacian performed for instance in \cite{nakasato1, donato1990, nakasato}, tools to deal with concentrated phenomenas and singular integrals introduced in \cite{gleice, arrieta, Nogueira,AMA-thin}, and iterated homogenization technique given by \cite{bunoiu,lions2001}.\\
	
	The homogenized equation of \eqref{problem_l} is the one dimensional $p$-Laplacian problem
	\begin{equation} \label{limint}
	\left\lbrace \begin{array}{ll}
	-q(|u'|^{p-2}u')'+|u|^{p-2}u=\bar{f} \textrm{ in }(0,1)\\
	u'(0)=u'(1)=0.
	\end{array}\right.
	\end{equation}
	We see that forcing term $\bar{f}$ is set by convergence assumptions on $f^\varepsilon$ and the functional 
	\begin{equation} \label{intc}
	\varphi \in W^{1,p}(R^\varepsilon) \mapsto \dfrac{1}{\varepsilon^\gamma}\int_{\mathcal{O}_\varepsilon}f^\varepsilon\varphi dxdy
	\end{equation}
	as $\varepsilon \to 0$. On the other side, the homogenized coefficient $q$ depends on parameter $\alpha$ by   
	\begin{equation*}\label{qvalue}
	q=\left\lbrace
	\begin{gathered}
	\dfrac{1}{\langle g\rangle_{(0,L_{g})}\langle 1/g^{p'-1}\rangle_{(0,L_{g})}^{p-1}} \quad \mbox{ if }0<\alpha<1\\
	\dfrac{1}{|Y^*|}\displaystyle\int_{Y^*}|\nabla v|^{p-2}\partial_{y_1}v\,dy_1dy_2 \quad \mbox{ if }\alpha=1\\
	\dfrac{g_{0}}{\langle g\rangle_{(0,L_{g})}} \quad \mbox{ if }\alpha>1
	\end{gathered}
	\right.
	\end{equation*}
	where $v$ is an auxiliary function which is the unique solution of the following problem
	\begin{align} \label{auxprob}
	\int_{Y^*}|\nabla v|^{p-2}\nabla v \nabla \varphi\, dy_1dy_2=0 \ \forall\varphi\in W_{\#}^{1,p}(Y^*) \\
	(v-y_1)\in W_{\#}^{1,p}(Y^*) \text{ with } \langle(v-y_1)\rangle_{Y^*}=0. \nonumber
	\end{align}
	Here we are setting the space $W_{\#}^{1,p}(Y^*)=\{\varphi\in W^{1,p}(Y^*) \; ; \; \varphi|_{\partial_{left}Y^*}=\varphi|_{\partial_{right}Y^*} \}$ and denoting $\left\langle \varphi\right\rangle_{\mathcal{U}}$ the average of $\varphi \in L^1_{loc}(\mathbb{R}^2)$ on open sets $\mathcal{U} \subset \mathbb{R}^2$.

Due to Minty-Browder's Theorem, one can show that problem \eqref{auxprob} is well posed, and then, the constant $q$ is well defined and positive for any $\alpha>0$ (see \cite{nakasato1, nakasato}).
	 If $\alpha<1$, the homogenized coefficient $q$ depends on $p \in (1,\infty)$, which establishes the order of the $p$-Laplacian operator, and on function $g$, which sets the profile of $R^\varepsilon$. 	
	At the case $\alpha>1$,  the coefficient $q$ does not depend explicitly on $p$, but just on the average of $g$ and its minimum value $g_0$, which is strictly positive. \\

    Next, we address a bounded nonlinear perturbation of the $p$-Laplacian. We consider  \\
	\begin{equation}\label{problem_s}
	\left\lbrace\begin{array}{lll}
	-\Delta_p u_\varepsilon+|u_\varepsilon|^{p-2}u_\varepsilon=\dfrac{1}{\varepsilon^\gamma}\chi_{\mathcal{O}^\varepsilon}f(u^\varepsilon)\mbox{ in }R^\varepsilon\\[0.3cm]
		|\nabla u_\varepsilon|^{p-2}\dfrac{\partial u_\varepsilon}{\partial \nu^\varepsilon}=0\mbox{ on }\partial R^\varepsilon
	\end{array}\right.
	\end{equation}
	assuming $p\geq 2$ and
	
	\par\medskip 	
	
	($\mathbf{H_f}$) $f\in C^2(\mathbb{R})$ is a bounded function with bounded derivatives.
	
	\par\medskip 
	
	As we will see, its homogenized equation is quite similar to \eqref{limint}.  We get the same homogenized coefficient $q$, but also an additional coefficient on the nonlinear reaction term. This former captures the oscillatory behavior and geometry of the thin domain and the narrow strip. The effective equation for this case is
	\begin{equation}\label{limite_s}
	\left\lbrace \begin{array}{ll}
	-q(|u'|^{p-2}u')'+|u|^{p-2}u=\dfrac{\langle h \rangle_{(0,L_h)} }{\langle g \rangle_{(0,L_g)} }f(u) \textrm{ in }(0,1),\\
	u'(0)=u'(1)=0.
	\end{array}\right.
	\end{equation}

	In the analysis of both problems \eqref{problem_l} and \eqref{limite_s}, we restrict the oscillatory behavior of the narrow strip by the oscillation parameter $\alpha$ of $R^\varepsilon$ dealing with $0<\beta<\alpha$. It is due to technical issues set by our iterated homogenization method adapted from \cite{bunoiu,lions2001} (see Proposition \ref{Propaux01} and Remark \ref{bmam0} below) in order to deal with concentrated integrals as \eqref{intc}. Anyway, the results still holds on conditions $0<\alpha<\beta$ (if we interchange the hypothesis of smoothness on $g$ and $h$), and on $\alpha = \beta>0$ under additional assumption $L_g = \kappa L_h$, with $\kappa\in\mathbb{N}$, on the period of functions $g$ and $h$. However, as $\beta=\alpha=0$, no oscillatory behavior is observed, and then, the analysis is a direct consequence of results from \cite{arrieta, SMarcone, MRi2}.

	The paper is organized as follows: in Section \ref{pre}, we state some notations and basic results introducing the iterated unfolding needed in the proofs. In Section \ref{linear}, we analyze problem \eqref{problem_l} for the three cases, $\alpha=1$, $0<\alpha<1$ and $\alpha>1$, obtaining the homogenized equation. In Section \ref{nonlinear}, we deal with the nonlinear perturbation of \eqref{problem_l} given by \eqref{limite_s}.

	\section{Preliminaries} \label{pre}
	
	In this section, we introduce some notations and recall some results concerning to monotone operators and the unfolding operator method. Next, we introduce a kind of reiterated homogenization approach for concentrated integrals proving some properties which will be useful in our analysis. 
	
	Recall we consider a two-dimensional family of thin domains exhibiting oscillatory behavior at its top boundary. 
	Let $Y^*$ given by
		\begin{equation*} 
		Y^*=\left\lbrace \right(y_1,y_2) \in \mathbb{R}^2 : 0<y_1<L_g\mbox{ and }0<y_2<g(y_1)\rbrace,
		\end{equation*}
	the basic cell of the thin domain $R^\varepsilon$ and 
	$$
	\left\langle \varphi\right\rangle_{\mathcal{O}} := \frac{1}{|\mathcal{O}|} \int_{\mathcal{O}} \varphi(x) \, dx
	$$
	the average of $\varphi \in L^1_{loc}(\mathbb{R}^2)$ on any open bounded set $\mathcal{O} \subset \mathbb{R}^2$. 
	
	We will also need to consider the following functional spaces which are defined by periodic functions in the variable $y_1 \in (0,L_g)$. Namely, 
	$$
	\begin{gathered}
	L^p_\#(Y^*) = \{ \varphi \in L^p(Y^*) \, : \, \varphi(y_1,y_2) \textrm{ is $L_g$-periodic in $y_1$ } \}, \\
	L^p_\#\left((0,1)\times Y^*\right) =  \{ \varphi \in L^p((0,1) \times Y^*) \, : \, \varphi(x, y_1,y_2) \textrm{ is $L_g$-periodic in $y_1$ } \}, \\
	W_{\#}^{1,p}(Y^*) = \{ \varphi\in W^{1,p}(Y^*) \, : \,  \varphi |_{\partial_{left} Y^*} = \varphi|_{\partial_{right} Y^*}\}.
	\end{gathered}
	$$
	
	If we denote by $[a]_L$ the unique integer number such that $a=[a]_L L+\{a\}_L$ where $\{a\}_L\in [0,L)$, then for each $\varepsilon>0$ and any $x\in\mathbb{R}$ we have 
	\begin{equation*}
	x=\varepsilon^\alpha \left[\frac{x}{\varepsilon^\alpha}\right]_{L_g}L_g+\varepsilon^\alpha\left\{\frac{x}{\varepsilon^\alpha}\right\}_{L_g}\mbox{ where }\left\{\frac{x}{\varepsilon^\alpha}\right\}_{L_g}\in [0,L_g)
	\end{equation*}
	with 
	\begin{equation} \label{Ox}
	x - \varepsilon^\alpha \left[\frac{x}{\varepsilon^\alpha}\right]_{L_g}L_g+\varepsilon^\alpha\left\{\frac{x}{\varepsilon^\alpha}\right\}_{L_g} = O(\varepsilon^\alpha).
	\end{equation}
	
	Let us also denote 
	\begin{equation*}
	I_\varepsilon=\mbox{ Int }\left(\bigcup_{k=0}^{N_\varepsilon}\left[kL_g\varepsilon^\alpha,(k+1)L_g\varepsilon^\alpha\right]\right), 
	\end{equation*}
	where $N_\varepsilon$ is the largest integer such that $\varepsilon^\alpha L_g(N_\varepsilon+1)\leq 1$,  
	$$
	\begin{gathered}
	\Lambda_\varepsilon=(0,1)\backslash I_\varepsilon=[\varepsilon^\alpha L_g(N_\varepsilon+1),1), \\
	R^\varepsilon_0=\left\lbrace(x,y)\in \mathbb{R}^2:x\in I_\varepsilon, 0<y<\varepsilon  g\left(\frac{x}{\varepsilon^\alpha}\right)\right\rbrace, \\
	R^\varepsilon_1=\left\lbrace(x,y)\in \mathbb{R}^2:x\in \Lambda_\varepsilon, 0<y<\varepsilon  g\left(\frac{x}{\varepsilon^\alpha}\right)\right\rbrace.
	\end{gathered}
	$$
	
	\begin{remark} Observe that, if $\varepsilon^\alpha L_g(N_\varepsilon+1)=1$, then $\Lambda_\varepsilon=\emptyset$. Consequently,  $R_0^\varepsilon=R^\varepsilon$ and $R_1^\varepsilon=\emptyset$. 
	\end{remark}

	Now, let us recall a result from \cite{cringanu} concerning to the $p$-Laplacian.
	
	\begin{proposition}\label{propositionduallitymap}
		Let $\Omega$ be an open bounded of $\mathbb{R}^n$. The duality mapping $J:W^{1,p}(\Omega)\to (W^{1,p}(\Omega))^*$ defined by
		$$
			(Ju,v)=\int_\Omega |\nabla u|^{p-2}\nabla u\nabla v +|u|^{p-2}u v\,dx,\quad\forall u,v\in W^{1,p}(\Omega),
		$$
		is single valued, bijective, with inverse bounded, continuous and monotone.
	\end{proposition}
%
	
	Next we recall the definition to the unfolding operator and some of its properties. For proofs and details see \cite{AM,AM2}.
	
	\begin{definition}
		Let $\varphi$ a Lebesgue-measurable function in $R^\varepsilon$. The unfolding operator $\mathcal{T}_\varepsilon$ acting on $\varphi$ is defined as the following function in $(0,1)\times Y^*$
		\begin{eqnarray*}
			\mathcal{T}_\varepsilon\varphi(x,y_1,y_2)=\left\lbrace\begin{array}{ll}
				\varphi\left(\varepsilon^\alpha\left[\frac{x}{\varepsilon^\alpha}\right]_{L_g}L_g+\varepsilon^\alpha y_1,\varepsilon y_2\right)\mbox{, for }(x,y_1,y_2)\in I_\varepsilon\times Y^*,\\
				0, \mbox{ for }(x,y_1,y_2)\in \Lambda_\varepsilon\times Y^*.
			\end{array}\right.
		\end{eqnarray*}
	\end{definition}
	
	\begin{proposition}\label{unfoldingproperties}
		The unfolding operator satisfies the following properties:
		\begin{enumerate}
			\item $\mathcal{T}_\varepsilon$ is linear under sum and multiplication operations;
			\item Let $\varphi\in L^1(R^\varepsilon)$. Then,
			\begin{eqnarray*}
				& & \frac{1}{L_g}\int_{(0,1)\times Y^*}\mathcal{T}_\varepsilon(\varphi)(x,y_1,y_2)dxdy_1dy_2=\frac{1}{\varepsilon}\int_{R_0^\varepsilon}\varphi(x,y)dxdy\\
				&& \qquad \qquad = \frac{1}{\varepsilon}\int_{R^\varepsilon}\varphi(x,y)dxdy-\frac{1}{\varepsilon}\int_{R_1^\varepsilon}\varphi(x,y)dxdy;
			\end{eqnarray*}
			\item For all $\varphi \in L^p(R^\varepsilon)$, $\mathcal{T}_\varepsilon(\varphi)\in L^p\left((0,1)\times Y^*\right)$, $1\leq p\leq \infty$. Moreover
			\begin{equation*}
			\left|\left|\mathcal{T}_\varepsilon(\varphi)\right|\right|_{L^p\left((0,1)\times Y^*\right)}=\left(\frac{L_g}{\varepsilon}\right)^{\frac{1}{p}}\left|\left|\varphi\right|\right|_{L^p(R_0^\varepsilon)}\leq \left(\frac{L_g}{\varepsilon}\right)^{\frac{1}{p}}\left|\left|\varphi\right|\right|_{L^p(R^\varepsilon)}.
			\end{equation*}
			If $p=\infty$,
			\begin{equation*}
			\left|\left|\mathcal{T}_\varepsilon(\varphi)\right|\right|_{L^\infty\left((0,1)\times Y^*\right)}=\left|\left|\varphi\right|\right|_{L^\infty(R_0^\varepsilon)}\leq \left|\left|\varphi\right|\right|_{L^\infty(R^\varepsilon)};
			\end{equation*}
			\item For all $\varphi \in W^{1,p}(R^\varepsilon)$, $1\leq p\leq \infty$,
			\begin{equation*}
			\partial_{y_1}\mathcal{T}_\varepsilon(\varphi)=\varepsilon^\alpha \mathcal{T}_\varepsilon(\partial_{x}\varphi)\mbox{ and }\partial_{y_2}\mathcal{T}_\varepsilon(\varphi)=\varepsilon \mathcal{T}_\varepsilon(\partial_{y}\varphi)\mbox{ a.e. in }(0,1)\times Y^*;
			\end{equation*}
			\item \label{uci}
			Let $\left(\varphi_\varepsilon\right)$ be a sequence in $L^p(R^\varepsilon)$, $1<p\leq\infty$ with the norm $\left|\left|\left|\varphi_\varepsilon\right|\right|\right|_{L^p(R^\varepsilon)}$ uniformly bounded. 
			Then $\left(\varphi_\varepsilon\right)$ satisfies the unfolding criterion for integrals (u.c.i), that is, 
			\begin{equation*}
			\frac{1}{\varepsilon}\int_{R_1^\varepsilon}|\varphi_\varepsilon|dxdy\rightarrow 0.
			\end{equation*}			
			Furthermore, let $\left(\psi_\varepsilon\right)$ be a sequence in $L^{q}(R^\varepsilon)$, also with $\left|\left|\left|\psi_\varepsilon\right|\right|\right|_{L^{q}(R^\varepsilon)}$ uniformly bounded, $\frac{1}{p}+\frac{1}{q}=\frac{1}{r}$, with $r>1$. 
			Then, the product sequence $(\varphi_\varepsilon\psi_\varepsilon)$ satisfies (u.c.i).
			If we take $\phi\in L^{p'}(0,1)$ then the sequence $\varphi_\varepsilon\phi$ satisfies (u.c.i).
			\item Let $\varphi\in L^p(0,1)$, $1\leq p<\infty$. Then, 
			\begin{equation*}
			\mathcal{T}_\varepsilon\varphi\rightarrow \varphi\mbox{ strongly in }L^p\left((0,1)\times Y^*\right).
			\end{equation*}
			
			\item \label{conv_unf_inicio} 
			Let $(\varphi_\varepsilon)$ a sequence in $L^p(0,1)$, $1\leq p<\infty$, such that 
			\begin{equation*}
			\varphi_\varepsilon\rightarrow \varphi\mbox{ strongly in }L^p(0,1).
			\end{equation*}
			Then, 
			\begin{equation*}
			\mathcal{T}_\varepsilon\varphi_\varepsilon\rightarrow \varphi\mbox{ strongly in }L^p\left((0,1)\times Y^*\right).
			\end{equation*}\,
		\end{enumerate}
	\end{proposition}
	

	From now on, we use the following rescaled norms in the thin open sets
	\begin{eqnarray*}
		&&\left|\left|\left|\varphi\right|\right|\right|_{L^p(R^\varepsilon)}=\varepsilon^{-1/p}\left|\left|\varphi\right|\right|_{L^p(R^\varepsilon)},\forall\varphi\in L^p(R^\varepsilon), 1\leq p< \infty,\\
		&&\left|\left|\left|\varphi\right|\right|\right|_{W^{1,p}(R^\varepsilon)}=\varepsilon^{-1/p}\left|\left|\varphi\right|\right|_{W^{1,p}(R^\varepsilon)}\forall\varphi\in W^{1,p}(R^\varepsilon), 1\leq p< \infty.
	\end{eqnarray*} 
	For completeness, we may denote $\left|\left|\left|\varphi\right|\right|\right|_{L^\infty(R^\varepsilon)}=\left|\left|\varphi\right|\right|_{L^\infty(R^\varepsilon)}$.
	\\

	Now we apply a reiterated homogenization argument to introduce another unfolding operator to deal with the concentrated terms. 
	For this sake, let us set, for $x \in (0,1)$, 
	\begin{eqnarray}\label{Ystarx}
	Y^*_\varepsilon(x) & = & \left\lbrace (y_1,y_2) \in \mathbb{R}^2  \, : \, 0<y_1<L_g, \right. \nonumber  \\ 
	& & \qquad \qquad \left. g(y_1)-\varepsilon^\gamma h\left( \frac{(\varepsilon^\alpha \left[\frac{x}{\varepsilon^\alpha}\right]_{L_g} L_g+\varepsilon^\alpha y_1)  } {\varepsilon^\beta} \right)<y_2<g(y_1)\right\rbrace.
	\end{eqnarray}
	Notice that $Y^*_\varepsilon(x) \subset Y^*$ and it is kind of a narrow strip in the representative cell $Y^*$.
	
	Analogously as before, we also define
	\begin{equation*}
	I^h_\varepsilon=\mbox{ Int }\left(\bigcup_{k=0}^{N_\varepsilon^h}\left[kL_h\varepsilon^\beta,(k+1)L_h\varepsilon^\beta\right]\right), 
	\end{equation*}
	where $N_\varepsilon^h$ is the largest integer such that $\varepsilon^\beta L_h(N_\varepsilon^h+1)\leq 1$,  
	$$
	\begin{gathered}
	\Lambda_\varepsilon^h=(0,1)\backslash I^h_\varepsilon=[\varepsilon^\beta L_h(N^h_\varepsilon+1),1), \\
	O^\varepsilon_0=\left\lbrace(x,y)\in \mathbb{R}^2:x\in I_\varepsilon, \varepsilon\left(g\left(\frac{x}{\varepsilon^\alpha}\right)-\varepsilon^\gamma h\left(\frac{x}{\varepsilon^\beta}\right)\right) <y<\varepsilon  g\left(\frac{x}{\varepsilon^\alpha}\right)\right\rbrace, \\
	O^\varepsilon_1=\left\lbrace(x,y)\in \mathbb{R}^2:x\in \Lambda_\varepsilon, \varepsilon\left(g\left(\frac{x}{\varepsilon^\alpha}\right)-\varepsilon^\gamma h\left(\frac{x}{\varepsilon^\beta}\right)\right)<y<\varepsilon  g\left(\frac{x}{\varepsilon^\alpha}\right)\right\rbrace.
	\end{gathered}
	$$
	
	\begin{remark} Notice that $\Lambda^h_\varepsilon=\emptyset$ as $\varepsilon^\beta L_h(N^h_\varepsilon+1)=1$, and then  $R_0^\varepsilon=R^\varepsilon$, $R_1^\varepsilon=\emptyset$, $O_0^\varepsilon=O^\varepsilon$ and $O_1^\varepsilon=\emptyset$. 
	\end{remark}

	Then, 
	\begin{equation*}
	\begin{gathered}
	\dfrac{1}{L_g\varepsilon^\gamma}\int_{(0,1)\times Y^*_\varepsilon(x)}\mathcal{T}_\varepsilon \varphi dxdY= \dfrac{1}{L_g\varepsilon^\gamma}\int_{I_\varepsilon\times Y^*_\varepsilon(x)}\varphi\left(\varepsilon^\alpha\left[\dfrac{x}{\varepsilon^\alpha}\right]_{L_g}L_g+\varepsilon^\alpha y_1,\varepsilon y_2\right)dYdx\\
	= \dfrac{1}{L_g\varepsilon^\gamma}\sum_{k=0}^{N_\varepsilon-1}\int_{k\varepsilon^\alpha L_g}^{(k+1)\varepsilon^\alpha L_g}\int_{Y^*_\varepsilon(x)}\varphi\left(\varepsilon^\alpha\left[\dfrac{x}{\varepsilon^\alpha}\right]_{L_g}L_g+\varepsilon^\alpha y_1,\varepsilon y_2\right)dYdx\\=\dfrac{1}{L_g\varepsilon^\gamma}\sum_{k=0}^{N_\varepsilon-1}\int_{k\varepsilon^\alpha L_g}^{(k+1)\varepsilon^\alpha L_g}\int_0^{L_g}\int_{g(y_1)-\varepsilon^\gamma h\left(\dfrac{\varepsilon^\alpha k L_g+\varepsilon^\alpha y_1 }{\varepsilon^\beta}\right)}^{g(y_1)}\varphi\left(\varepsilon^\alpha kL_g+\varepsilon^\alpha y_1,\varepsilon y_2\right)dy_2dy_1dx\\=\dfrac{1}{\varepsilon^\gamma}\sum_{k=0}^{N_\varepsilon-1}\int_0^{L_g}\int_{g(y_1)-\varepsilon^\gamma h\left(\dfrac{\varepsilon^\alpha k L_g+\varepsilon^\alpha y_1 }{\varepsilon^\beta}\right)}^{g(y_1)}\varphi\left(\varepsilon^\alpha kL_g+\varepsilon^\alpha y_1,\varepsilon y_2\right)\varepsilon^\alpha dy_2dy_1.
	\end{gathered}
	\end{equation*} 
	Thus, performing the change of variables  
	$x=\varepsilon^\alpha kL_g+\varepsilon^\alpha y_1$ and $y= \varepsilon y_2$,
	we obtain
	\begin{align*}
	\dfrac{1}{\varepsilon^\gamma}\sum_{k=0}^{N_\varepsilon-1}\int_0^{L_g}\int_{g(y_1)-\varepsilon^\gamma h\left(\dfrac{\varepsilon^\alpha k L_g+\varepsilon^\alpha y_1 }{\varepsilon^\beta}\right)}^{g(y_1)}\varphi\left(\varepsilon^\alpha kL_g+\varepsilon^\alpha y_1,\varepsilon y_2\right)\varepsilon^\alpha dy_2dy_1 = \\  =\dfrac{1}{\varepsilon^{\gamma+1}}\sum_{k=0}^{N_\varepsilon-1}\int_{k\varepsilon^\alpha L_g}^{(k+1)\varepsilon^\alpha L_g} \int_{\varepsilon\left[g\left(\dfrac{x}{\varepsilon^\alpha}\right)-\varepsilon^\gamma h\left(\dfrac{\varepsilon^\alpha kL_g+\varepsilon^\alpha y_1}{\varepsilon^\beta}\right)\right]}^{\varepsilon g\left(\dfrac{x}{\varepsilon^\alpha}\right)} \varphi (x,y) dx dy \\  = \dfrac{1}{\varepsilon^{\gamma+1}}\int_{\mathcal{O}^\varepsilon_0}\varphi dxdy=\dfrac{1}{\varepsilon^{\gamma+1}}\int_{\mathcal{O}^\varepsilon}\varphi dxdy-\dfrac{1}{\varepsilon^{\gamma+1}}\int_{\mathcal{O}^\varepsilon_1}\varphi dxdy.
	\end{align*}
	
	Therefore, we can rewrite the concentrated integral in terms of the unfolding operator. Indeed, we have proved the following result:
	\begin{proposition}\label{first_unf}
		If $\varphi\in L^1(R^\varepsilon)$, then
		$$
		\dfrac{1}{\varepsilon^{\gamma+1}}\int_{\mathcal{O}^\varepsilon}\varphi dxdy=\dfrac{1}{L_g\varepsilon^\gamma}\int_{(0,1)\times Y^*_\varepsilon(x)}\mathcal{T}_\varepsilon \varphi dxdY+\dfrac{1}{\varepsilon^{\gamma+1}}\int_{\mathcal{O}^\varepsilon_1}\varphi dxdy.
		$$\,
	\end{proposition}
		
%
	Next, we see how we can rewrite $Y_\varepsilon^*(x)$ in a simpler way.
	\begin{proposition}\label{Propaux01} 
		
		
		If $\varphi(x,y_1,\cdot)\in C(0,g(y_1))$ a.e. in $(0,1)\times(0,L_g)$, then
		$$
		\dfrac{1}{\varepsilon^\gamma}\left|\int_{g(y_1)-\varepsilon^\gamma h\left(\frac{\varepsilon^\alpha [x/\varepsilon^\alpha]L_g+\varepsilon^\alpha y_1 }{\varepsilon^\beta}\right)}^{g(y_1)}\varphi(x,y_1,y_2)dy_2-	\int_{g(y_1)-\varepsilon^\gamma h\left(\frac{x}{\varepsilon^\beta}\right)}^{g(y_1)}\varphi(x,y_1,y_2)dy_2\right|\to 0
		$$
		a.e. in $(0,1)\times(0,L_g)$ as $\varepsilon\to 0$.
	\end{proposition}
	\begin{proof} 
		Let 
		$$
		\varrho_\varepsilon=\varrho_\varepsilon(x,y_1,\varepsilon)=\varepsilon^\alpha [x/\varepsilon^\alpha]L_g+\varepsilon^\alpha y_1,  
		\qquad h_\varepsilon(\cdot)=h\left(\dfrac{\cdot}{\varepsilon^\beta}\right).
		$$
		Suppose, without loss of generality, that 
		$$
		h\left(\frac{\varepsilon^\alpha [x/\varepsilon^\alpha]L_g+\varepsilon^\alpha y_1 }{\varepsilon^\beta}\right)\geq  h\left(\frac{x}{\varepsilon^\beta}\right).
		$$
		Notice that
		$$
		\begin{gathered}
		\dfrac{1}{\varepsilon^\gamma}\left|\int_{g(y_1)-\varepsilon^\gamma h_\varepsilon(\varrho_\varepsilon)}^{g(y_1)}\varphi(x,y_1,y_2)dy_2-	\int_{g(y_1)-\varepsilon^\gamma h_\varepsilon(x)}^{g(y_1)}\varphi(x,y_1,y_2)dy_2\right|
		= \dfrac{1}{\varepsilon^\gamma}\left|\int_{g(y_1)-\varepsilon^\gamma h_\varepsilon(\varrho_\varepsilon)}^{g(y_1)-\varepsilon^\gamma h_\varepsilon(x)}\varphi(x,y_1,y_2)dy_2\right| 
		\end{gathered}
		$$
		Consider the change of variables
		$$
		z_2=\dfrac{y_2-g(y_1)+\varepsilon^\gamma h_\varepsilon(\varrho_\varepsilon)}{\varepsilon^\gamma(h_\varepsilon(\varrho_\varepsilon)-h_\varepsilon(x))}, \qquad dz_2=\dfrac{dy_2}{\varepsilon^\gamma(h_\varepsilon(\varrho_\varepsilon)-h_\varepsilon(x))}. 
		$$
		Then,
		
		
		\begin{equation*}
		\begin{gathered}
		\dfrac{1}{\varepsilon^\gamma}\left|\int_{g(y_1)-\varepsilon^\gamma h_\varepsilon(\varrho_\varepsilon)}^{g(y_1)-\varepsilon^\gamma h_\varepsilon(x)}\varphi(x,y_1,y_2)dy_2\right| =\\
		=\left|\int_0^1\varphi\left(x,y_1,\varepsilon^\gamma(h_\varepsilon(\varrho_\varepsilon)-h_\varepsilon(x))z_2+g(y_1)-\varepsilon^\gamma h_\varepsilon(\varrho_\varepsilon)\right) (h_\varepsilon(\varrho_\varepsilon)-h_\varepsilon(x)) dz_2\right|\\
		\leq \int_0^1\left|\varphi\left(x,y_1,\varepsilon^\gamma(h_\varepsilon(\varrho_\varepsilon)-h_\varepsilon(x))z_2+g(y_1)-\varepsilon^\gamma h_\varepsilon(\varrho_\varepsilon)\right)\right| ||h'||_{\infty} \dfrac{\left|\varrho_\varepsilon-x\right|}{\varepsilon^\beta} dz_2.
		\end{gathered}
		\end{equation*}
		Now, since  $\alpha> \beta$ and $\varepsilon^\alpha [x/\varepsilon^\alpha]L_g+\varepsilon^\alpha y_1$ satisfies \eqref{Ox}, we have that
		$$
		\begin{gathered}
		\dfrac{1}{\varepsilon^\gamma}\left|\int_{g(y_1)-\varepsilon^\gamma h_\varepsilon(\varrho_\varepsilon)}^{g(y_1)-\varepsilon^\gamma h_\varepsilon(x)}\varphi(x,y_1,y_2)dy_2\right|\to 0,
		\end{gathered}
		$$
		as $\varepsilon\to 0$, proving the result.
	\end{proof}

\begin{remark}		One could also suppose that $h$ is uniformly continuous in Proposition \ref{Propaux01} with minor changes in the proof. \end{remark} 

	\begin{remark}
		By Proposition \ref{Propaux01}, we can rewrite \eqref{Ystarx}, without loss of generality, as
		\begin{equation}\label{Ystarx01}
		Y^*_\varepsilon(x)=\left\lbrace (y_1,y_2)\in\mathbb{R}^2:0<y_1<L_g,g(y_1) - \varepsilon^\gamma h\left(\dfrac{x}{\varepsilon^\beta} \right)<y_2<g(y_1)    \right\rbrace
		\end{equation}
		assuming $\varepsilon>0$ is small enough. 
	\end{remark}

	\begin{remark} \label{bmam0}
		In Proposition \ref{Propaux01}, we strongly use $\alpha>\beta$. To analyze $\beta>\alpha$, we need to define a different subset $Y_\varepsilon^*(x)$, modifying the change of variables in the proof of \ref{first_unf} which would lead us to a different unfolding operator. Since the analysis is analogous, we will focus here on $Y_\varepsilon^*$ defined in \eqref{Ystarx01} with $\alpha > \beta>0$. We point out that, in this other case, the hypothesis  on $g$ and $h$ would be interchanged.
	\end{remark}
\begin{remark}  Notice that case $\alpha = \beta>0$, we can perform the analysis when the periods of $g$ and $h$ are rationally dependent, that is, there exists $\kappa\in\mathbb{N}$ such that $L_h=\kappa L_g$.  In this case, we do not even need results related to the iterated unfolding technique we describe bellow.
\end{remark} 

	We also have the following results.
	\begin{proposition}\label{bound}
		If $\varphi_\varepsilon\in W^{1,p}(R^\varepsilon)$, there exists $C>0$ independent of $\varepsilon$ such that, for $1-1/p<s\leq 1$,
		$$
		\dfrac{1}{\varepsilon^{\gamma}}\int_{\mathcal{O}^\varepsilon}|\varphi_\varepsilon|^qdxdy 
		\leq C \|\varphi_\varepsilon\|^q_{L^q((0,1);W^{s,p}(0,\varepsilon g(x/\varepsilon^\alpha))}, \ \forall q\geq 1. 
		$$
		In particular, if $p\geq2$ we have
		$$
		\dfrac{1}{\varepsilon^{\gamma}}\int_{\mathcal{O}^\varepsilon}|\varphi_\varepsilon|^q dxdy\leq C \|\varphi_\varepsilon\|^q_{W^{1,p}(R^\varepsilon)}, \ \forall q\leq p. $$
		\end{proposition}
	\begin{proof}
 The first part is true by changing the $H^s$ space by $W^{s,p}$, for $1-1/p<s\leq 1$, in \cite[Theorem 3.7]{AMA-thin}, which implies that exists $C>0$ independent of $\varepsilon>0$ such that
	$$
	\dfrac{1}{\varepsilon^{\gamma}}\int_{\mathcal{O}^\varepsilon}|\varphi_\varepsilon|^qdxdy\leq C  \|\varphi_\varepsilon\|^q_{L^q(0,1;W^{s,p}(0,\varepsilon g(x/\varepsilon^\alpha)))}, \ \forall q\geq 1. 
	$$
	Also, it is not difficult to see that $W^{1,p}(R^\varepsilon)$ is included in $L^p(0,1;W^{1,p}(0,\varepsilon g(x/\varepsilon^\alpha)))$ and, consequently, in $L^p(0,1;W^{s,p}(0,\varepsilon g(x/\varepsilon^\alpha)))$ with constant independent of $\varepsilon>0$ analogously as \cite[Proposition 3.6]{AMA-thin}.
	
	On the other hand, if $p>2$, we have that $\frac{p}{q}>1$ for all $q< p$ and then, if we call $g_\varepsilon(x)=\varepsilon g(x/\varepsilon^\alpha)$,
	\begin{align*}
	\|\varphi_\varepsilon\|&^q_{L^q(0,1;W^{1,p}(0,g_\varepsilon(x)))}  = \int_0^1\left(\int_0^{g_\varepsilon(x)}|\varphi_\varepsilon(x,y)|^pdy\right)^{\frac{q}{p}}dx + \int_0^1\left(\int_0^{g_\varepsilon(x)}|\partial_x \varphi_\varepsilon(x,y)|^pdy\right)^{\frac{q}{p}}dx \\ &  \leq g_1^{\frac{p-q}{p}}\left(\int_0^1\int_0^{g_\varepsilon(x)}|\varphi_\varepsilon(x,y)|^pdxdy\right)^{\frac{q}{p}} + g_1^{\frac{p-q}{p}}\left(\int_0^1\int_0^{g_\varepsilon(x)}|\partial_x \varphi_\varepsilon(x,y)|^pdxdy\right)^{\frac{q}{p}} \leq C\|\varphi_\varepsilon\|^q_{W^{1,p}(R^\varepsilon)}
	\end{align*}		
\end{proof}

	Now, we proceed proving u.c.i. for integrals in $(0,1)\times Y^*_\varepsilon(x)$. 
	
	\begin{proposition}\label{uciO1}
		If $\varphi_\varepsilon\in W^{1,p}(R^\varepsilon)$ is uniformly bounded in the norm $|||\cdot|||_{W^{1,p}(R^\varepsilon)}$, then
		$$
		\dfrac{1}{\varepsilon^{\gamma+1}}\int_{\mathcal{O}^\varepsilon_1}|\varphi_\varepsilon| dxdy\to 0.
		$$
	\end{proposition}
	\begin{proof}
		It follows from 
		\begin{equation*}
		\begin{gathered}
		\dfrac{1}{\varepsilon^{\gamma+1}}\int_{\mathcal{O}^\varepsilon_1}|\varphi_\varepsilon| dxdy\leq \left(\dfrac{1}{\varepsilon^{\gamma+1}}\int_{\mathcal{O}^\varepsilon_1}|\varphi_\varepsilon|^p dxdy\right)^{1/p} \left(\dfrac{1}{\varepsilon^{\gamma+1}}\int_{\mathcal{O}^\varepsilon_1} 1 dxdy\right)^{1/p'}\\
		\leq C |||\varphi_\varepsilon|||_{W^{1,p}(R^\varepsilon)}  |\Lambda_\varepsilon|^{1/p'}
		\end{gathered}
		\end{equation*}
		since $\varphi_\varepsilon$ is uniformly bounded and $|\Lambda_\varepsilon| \to 0$ as $\varepsilon \to 0$ by Proposition \ref{bound}. 
	\end{proof}

	The next result is adapted from \cite[Theorem 3.7]{AMA-thin}. Again, we use a Lebesgue-Bochner space to analyze the concentrated integrals.

	\begin{proposition}\label{propboundnessuci}
		The following properties are satisfied:
			\item{(a)} If $\varphi \in L^p\left((0,1);W^{1,p}(Y^*)\right)$, then $$\dfrac{1}{\varepsilon^\gamma}\displaystyle\int_{(0,1)\times Y^*_\varepsilon(x)} |\varphi|^p dx dY\leq C||\varphi||^p_{L^p\left((0,1);W^{1,p}(Y^*)\right)},$$
			with $C>0$ independent of $x \in (0,1)$ and $\varepsilon>0$. \\
			\item{(b)} (u.c.i.) If $(\varphi_\varepsilon)\subset  L^p\left((0,1);W^{1,p}(Y^*)\right)$ is uniformly bounded, then
			$$
			\dfrac{1}{\varepsilon^\gamma}\displaystyle\int_{\Lambda_\varepsilon^h\times Y^*_\varepsilon(x)} |\varphi_\varepsilon| dx dY\to 0.
			$$
	\end{proposition}
	\begin{proof}
		\textit{(a) } For $\varphi\in L^p\left((0,1);W^{1,p}(Y^*)\right)$, we have $\varphi(x,y_1,\cdot)\in W^{1,p}(0,g(y_1))$ a.e. $(0,1)\times (0,L_g)$. Define 
		$$
		z^\varepsilon_1=g(y_1)-\varepsilon^\gamma h_1\mbox{ and }z^\varepsilon_0=g(y_1)-\varepsilon^\gamma h(x/\varepsilon^\beta)
		$$
		for $h_1 = \max_{x \in \mathbb{R}} h(x)$ and $\varepsilon>0$ small enough such that 
		$$
		[z^\varepsilon_0-z^\varepsilon_1,z^\varepsilon_0]\subset [0,g(y_1)].
		$$
		Notice that since $g(y_1)-\varepsilon^\gamma h(x/\varepsilon^\beta)<y_2<g(y_1)$, for $1-1/p<s\leq 1$, it follows from \cite[Theorem 1.5.1.3]{grisvard} for $n=1$ that there exists $K>0$ independent of $\varepsilon$ such that 
		$$
		|\varphi(x,y_1,y_2)|\leq K ||\varphi(x,y_1,\cdot)||_{W^{s,p}(y_2-z^\varepsilon_0,y_2)}\leq K ||\varphi(x,y_1,\cdot)||_{W^{s,p}(0,g(y_1))}.
		$$
		Hence,
		\begin{equation*}
		\begin{gathered}
		\dfrac{1}{\varepsilon^\gamma}\displaystyle\int_{(0,1)\times Y^*_\varepsilon(x)} |\varphi|^p dx dY\leq \int_0^1\int_0^{L_g}\dfrac{1}{\varepsilon^\gamma}\int_{g(y_1)-\varepsilon^\gamma h(x/\varepsilon^\beta)}^{g(y_1)} |\varphi(x,y_1,y_2)|^p dy_2 dy_1 dx\\
		\leq  \int_0^1\int_0^{L_g}\dfrac{1}{\varepsilon^\gamma}\int_{g(y_1)-\varepsilon^\gamma h(x/\varepsilon^\beta)}^{g(y_1)} K^p ||\varphi(x,y_1,\cdot)||^p_{W^{s,p}(0,g(y_1))} dy_2 dy_1 dx\\
		\leq C\int_0^1\int_0^{L_g}||\varphi(x,y_1,\cdot)||^p_{W^{s,p}(0,g(y_1))} dy_1 dx=C\int_0^1|| \varphi||^p_{L^p((0,L_g);W^{s,p}(0,g(y_1)))}dx\\
		\leq \bar{C}\int_0^1|| \varphi||^p_{W^{1,p}(Y^*)}dx\leq \bar{C}|| \varphi||^p_{L^p((0,1);W^{1,p}(Y^*))},
		\end{gathered}
		\end{equation*}
		where $W^{1,p}(Y^*)\subseteq L^p((0,L_g);W^{1,p}(0,g(y_1)))\subset L^p((0,L_g);W^{s,p}(0,g(y_1)))$.\\
		
		\noindent\textit{(b) } Using H\"older's inequality and item (a), one gets
		$$
		\dfrac{1}{\varepsilon^\gamma}\displaystyle\int_{\Lambda_\varepsilon^h\times Y^*_\varepsilon(x)} |\varphi_\varepsilon| dx dY\leq C|\Lambda_\varepsilon^h|^{1/p'}\cdot|| \varphi_\varepsilon||_{L^p((0,1);W^{1,p}(Y^*))}\to 0.
		$$
	\end{proof}

	Now, we are in conditions to introduce the following unfolding operator:
	\begin{definition}\label{unfoldingiterated}
		Let $\varphi$ be a measurable function in $(0,1)\times Y^*_\varepsilon(x)$. We define the unfolding operator $\mathcal{T}_\varepsilon^g$ as 
		\begin{eqnarray*}
			\mathcal{T}_\varepsilon^g\varphi(x,y_1,z_1,z_2)=\left\lbrace\begin{array}{ll}
				\varphi\left(\varepsilon^\beta\left[\frac{x}{\varepsilon^\beta}\right]_{L_h}L_h+\varepsilon^\beta z_1,y_1,\varepsilon^\gamma z_2+g(y_1)\right),\\ \qquad \qquad \qquad \mbox{ for }(x,y_1,z_1,z_2)\in I_\varepsilon^h\times (0,L_g)\times Y^*_h,\\[0.15cm]
				0, \mbox{ for }(x,y_1,z_1,z_2)\in \Lambda_\varepsilon^h\times  (0,L_g)\times Y^*_h
			\end{array}\right.
		\end{eqnarray*}
	where $Y^*_h=\left\lbrace \right(z_1,z_2) \in \mathbb{R}^2 : 0<z_1<L_h\mbox{ and }-h(z_1)<z_2<0\rbrace$.
	\end{definition}
	
	\begin{remark}
	\begin{itemize}
	\item[(a)] Notice that $\mathcal{T}_{\varepsilon}^{g}$ preserves the regularity of $\varphi$ as a function on $y_{1}$ variable.
	
	\item[(b)] The regularity of $\mathcal{T}_{\varepsilon}^{g}\varphi$ in the variable $z_{2}$ is inherited from $y_{2}$.  
	
	\item[(c)] Let $\psi$ be a function defined in $R^{\varepsilon}$. We know that $\mathcal{T}_{\varepsilon}\psi$ inherits the regularity from $\psi$ as a function of variable $(x,y)$ into variable $(y_{1},y_{2})$. Therefore $\mathcal{T}_{\varepsilon}^{g}(\mathcal{T}_{\varepsilon}\psi)$ will be a regular function on $y_{1}$ and $z_{2}$ variables. 
	\end{itemize}
	\end{remark} 
		Proceeding as in the proof of Proposition \ref{first_unf}, one can get
		\begin{equation*}
		\begin{gathered}
		\dfrac{1}{L_gL_h}\int_{(0,1)\times (0,L_g)\times Y^*_h}\mathcal{T}_\varepsilon^g(\mathcal{T}_\varepsilon \varphi) dxdy_1 dZ\\=\dfrac{1}{L_gL_h} \sum_{k=0}^{N_\varepsilon-1}\int_{k\varepsilon^\beta L_h}^{(k+1)\varepsilon^\beta L_h}\int_{0}^{L_g}\int_{0}^{L_h}\int_{-h(z_1)}^0 \mathcal{T}_\varepsilon \varphi \left(\varepsilon^\beta kL_h+\varepsilon^\beta z_1,y_1,\varepsilon^\gamma z_2+g(y_1)\right) dz_2 dz_1dy_1dx\\=\dfrac{L_h\varepsilon^\beta}{L_gL_h} \sum_{k=0}^{N_\varepsilon-1}\int_{0}^{L_g}\int_{0}^{L_h}\int_{-h(z_1)}^0 \mathcal{T}_\varepsilon \varphi \left(\varepsilon^\beta kL_h+\varepsilon^\beta z_1,y_1,\varepsilon^\gamma z_2+g(y_1)\right) dz_2 dz_1dy_1.
		\end{gathered}
		\end{equation*} 
			Hence, performing the change of variables 
			$$\begin{cases}
			x=\varepsilon^\beta kL_h+\varepsilon^\beta z_1 \\y_1 = y_1 \\ y_2 = \varepsilon^\gamma z_2 + g(y_1) 
			\end{cases}\Rightarrow  dy_2dxdy_1 = \varepsilon^{\beta+\gamma}dz_2dz_1dy_1$$
			we obtain
		\begin{equation*}
		\begin{gathered}  
		\dfrac{\varepsilon^\beta}{L_g} \sum_{k=0}^{N_\varepsilon-1}\int_{0}^{L_g}\int_{0}^{L_h}\int_{-h(z_1)}^0 \mathcal{T}_\varepsilon \varphi \left(\varepsilon^\beta kL_h+\varepsilon^\beta z_1,y_1,\varepsilon^\gamma z_2+g(y_1)\right) dz_2 dz_1dy_1	\\ = \dfrac{1}{\varepsilon^\gamma L_g} \sum_{k=0}^{N_\varepsilon-1}\int_{0}^{L_g}\int_{k\varepsilon^\beta L_h}^{(k+1)\varepsilon^\beta L_h}\int_{g(y_1)-\varepsilon^\gamma h(x/\varepsilon^\beta)}^{g(y_1)} \mathcal{T}_\varepsilon \varphi \left(x,y_1,y_2\right) dy_2 dx dy_1	\\ = \dfrac{1}{\varepsilon^\gamma L_g}\int_{(0,1)\times Y^*_\varepsilon(x)}\mathcal{T}_\varepsilon\varphi dx dY-\dfrac{1}{\varepsilon^\gamma L_g}\int_{\Lambda_\varepsilon^h\times Y^*_\varepsilon(x)}\mathcal{T}_\varepsilon\varphi dx dY.
		\end{gathered}
		\end{equation*}\\
	
		Consequently, we obtain the following propositions.			
			\begin{proposition}\label{igualdade}
				If $\varphi$ be a measurable function in $(0,1)\times  Y^*_\varepsilon(x)$ then
				\begin{align*}\dfrac{1}{L_gL_h}\int_{(0,1)\times (0,L_g)\times Y^*_h}&\mathcal{T}_\varepsilon^g(\mathcal{T}_\varepsilon \varphi) dxdy_1 dZ = \\ = & \dfrac{1}{\varepsilon^\gamma L_g}\int_{(0,1)\times Y^*_\varepsilon(x)}\mathcal{T}_\varepsilon\varphi dx dY-\dfrac{1}{\varepsilon^\gamma L_g}\int_{\Lambda_\varepsilon^h\times Y^*_\varepsilon(x)}\mathcal{T}_\varepsilon\varphi dx dY
				\end{align*}
			\end{proposition}
		
	\begin{proposition}\label{iteration}
		If $\varphi\in L^1(R^\varepsilon)$, then
		\begin{align*}
		\dfrac{1}{\varepsilon^{\gamma+1}}\int_{\mathcal{O}^\varepsilon}\varphi dxdy
		= \dfrac{1}{L_gL_h}\int_0^1& \int_0^{L_g}\int_{Y^*_h}\mathcal{T}_\varepsilon^g(\mathcal{T}_\varepsilon \varphi) dZdy_1dx \ + \\ & +\dfrac{1}{\varepsilon^\gamma L_g}\int_{\Lambda_\varepsilon^h\times Y^*_\varepsilon(x)}\mathcal{T}_\varepsilon\varphi dx dY+\dfrac{1}{\varepsilon^{\gamma+1}}\int_{\mathcal{O}^\varepsilon_1}\varphi dxdy.
		\end{align*}
	\end{proposition}

	\begin{proposition}\label{unfoldingpropertiesg}
		Let $\varphi\in L^{1}((0,1)\times Y^{*}_{\varepsilon}(x))$. Then
		\begin{enumerate}
		\item[(a)] $\mathcal{T}_{\varepsilon}^{g}$ is linear under sum and multiplication operations.
		\item[(b)] Let $f:\mathbb{R}\to\mathbb{R}$. Then, $\mathcal{T}_{\varepsilon}^{g}(f(\varphi))=f(\mathcal{T}_{\varepsilon}^{g}\varphi)$.
		\item[(c)] Let $\varphi\in L^{p}((0,1)\times Y^{*}_{\varepsilon}(x))$. Then,
		$$
		||\mathcal{T}_{\varepsilon}^{g} \varphi||_{L^{p}((0,1)\times (0,L_{g})\times Y^{*}_{h})}\leq\dfrac{c}{\varepsilon^{\gamma}}||\varphi||_{L^{p}((0,1)\times Y^{*}_{\varepsilon}(x))}
		$$
		\end{enumerate}
	\end{proposition}
	\begin{proof}
	The proof follows directly from Definition \ref{unfoldingiterated}.
	\end{proof}
	
	Furthermore, we can obtain the following convergence results.

	\begin{proposition}\label{propconvtestfunction}
		Let $\varphi\in L^{p}(0,1)$. Then,
		$$
		\mathcal{T}_\varepsilon^g\varphi \to \varphi
		$$
		strongly in $L^p((0,1)\times(0,L_g)\times Y^*_h)$.
	\end{proposition}
	\begin{proof}
		The proof is simple and can be performed as in \cite[Proposition 2.11]{AM2}.
	\end{proof}
	\begin{proposition}\label{conv_unf}
		Let $\varphi_\varepsilon\in L^p((0,1);W^{1,p}(Y^*))$ such that $\varphi_\varepsilon\to \varphi$ in $L^p((0,1);W^{1,p}(Y^*))$, where $\varphi\in W^{1,p}(0,1)$. Then
		$$
		\mathcal{T}_\varepsilon^g\varphi_\varepsilon\to \varphi
		$$
		strongly in $L^p((0,1)\times(0,L_g)\times Y^*_h)$.
	\end{proposition}
	\begin{proof}
		\begin{equation}
		\begin{gathered}
		||\mathcal{T}_\varepsilon^g\varphi_\varepsilon- \varphi||_{L^p((0,1)\times(0,L_g)\times Y^*_h)}\\\leq ||\mathcal{T}_\varepsilon^g\varphi_\varepsilon- \mathcal{T}_\varepsilon^g\varphi||_{L^p((0,1)\times(0,L_g)\times Y^*_h)}+||\mathcal{T}_\varepsilon^g\varphi-\varphi||_{L^p((0,1)\times(0,L_g)\times Y^*_h)}.
		\end{gathered}
		\end{equation}
		
		Using Proposition \ref{propconvtestfunction} we obtain
		$$||\mathcal{T}_\varepsilon^g\varphi-\varphi||_{L^p((0,1)\times(0,L_g)\times Y^*_h)}\to 0$$ as $\varepsilon\to 0$. 
		Also, by Propositions  \ref{propboundnessuci} and \ref{igualdade}, we have 
		$$
		\begin{gathered}
		||\mathcal{T}_\varepsilon^g\varphi_\varepsilon- \mathcal{T}_\varepsilon^g\varphi||_{L^p((0,1)\times(0,L_g)\times Y^*_h)}^p\leq \dfrac{C}{\varepsilon^\gamma}\int_{(0,1)\times Y^*_\varepsilon(x)} |\varphi_\varepsilon-\varphi|^p dxdY-\dfrac{C}{\varepsilon^\gamma}\int_{\Lambda_\varepsilon^h\times Y^*_\varepsilon(x)}|\varphi_\varepsilon-\varphi|^p dxdY \\
		\leq C||\varphi_\varepsilon-\varphi||^p_{L^p((0,1);W^{1,p}(Y^*))}-\dfrac{C}{\varepsilon^\gamma}\int_{\Lambda_\varepsilon^h\times Y^*_\varepsilon(x)}|\varphi_\varepsilon-\varphi|^p dxdY \to 0.
		\end{gathered}
		$$
		Therefore, one concludes that
		$$
		\mathcal{T}_\varepsilon^g\varphi_\varepsilon\to \varphi
		$$
		strongly in $L^p((0,1)\times(0,L_g)\times Y^*_h)$.
	\end{proof}
	
	The next results are direct consequence of Proposition \ref{conv_unf}.
	\begin{corollary}
		Let $\varphi_\varepsilon\in W^{1,p}(R^\varepsilon)$ such that there is $\varphi\in W^{1,p}(0,1)$ satisfying
		$$
		\mathcal{T}_\varepsilon \varphi_\varepsilon\to\varphi\mbox{ strongly in }L^p((0,1);W^{1,p}(Y^*)).
		$$
		Then
		$$
		\mathcal{T}_\varepsilon^g(\mathcal{T}_\varepsilon\varphi_\varepsilon)\to \varphi\mbox{ strongly in }L^p((0,1)\times (0,L_g)\times Y^*_h).
		$$\,
	\end{corollary}

	\begin{corollary}\label{conv_lin}
		Let $f^\varepsilon\in L^{p'}(R^\varepsilon)$ such that
		$$
		\mathcal{T}_\varepsilon^g(\mathcal{T}_\varepsilon f^\varepsilon) \rightharpoonup \hat{f}\mbox{ weakly in }L^{p'}((0,1)\times(0,L_g)\times Y^*_h)
		$$
		for some $\hat{f}\in L^{p'}(0,1)$. Then,
		$$
		\dfrac{1}{L_gL_h}\int_0^1\int_0^{L_g}\int_{Y^*_h} \mathcal{T}_\varepsilon^g(\mathcal{T}_\varepsilon f^\varepsilon)\mathcal{T}_\varepsilon^g(\mathcal{T}_\varepsilon \varphi) dxdy_1dZ \to\dfrac{1}{L_gL_h} \int_0^1\int_0^{L_g}\int_{Y^*_h}\hat{f}\varphi dxdy_1dZ
		$$
		for all $\varphi\in L^p(0,1)$. Moreover,
		$$
		\dfrac{1}{\varepsilon^{\gamma+1}}\int_{\mathcal{O}^\varepsilon} f^\varepsilon\varphi dxdy\to \int_0^1 \bar{f}
		\varphi dx,
		$$
		with 
		$$
		\bar{f}=\dfrac{1}{L_gL_h}\int_{(0,L_g)\times Y^*_h} \hat{f} dy_1 dZ.
		$$
	\end{corollary}
	\begin{remark}
		If $f^\varepsilon(x,y)=f(x)$, then
		$$
			\dfrac{1}{\varepsilon^{\gamma+1}}\int_{\mathcal{O}^\varepsilon} f^\varepsilon\varphi dxdy\to \langle h\rangle _{(0,L_h)}\int_0^1 f \varphi dx.
		$$\,
	\end{remark}
	
	Now, we consider concentrated integrals with nonlinear terms given by bounded $C^2$-functions with bounded derivatives. 
	\begin{proposition}\label{conv_f}
		Let $f:\mathbb{R}\to \mathbb{R}$ be a $\mathcal{C}^2$ function satisfying ($\mathbf{H_f}$). Then,
		$$
		\int_0^1\int_0^{L_g}\int_{Y^*_h}\mathcal{T}_\varepsilon^g(\mathcal{T}_\varepsilon(f(\varphi_\varepsilon)))\mathcal{T}_\varepsilon^g\mathcal{T}_\varepsilon\psi \to \int_0^1\int_0^{L_g}\int_{Y^*_h}f(\varphi) \psi,\quad \forall \psi\in W^{1,p}(0,1),
		$$
		whenever $\varphi_\varepsilon\in W^{1,p}(R^\varepsilon)$ satisfies $\mathcal{T}_\varepsilon\varphi_\varepsilon\to \varphi$  in $L^p((0,1);W^{1,p}(Y^*))$. 
		
	\end{proposition}
	\begin{proof}
		See that
		$$
		\begin{gathered}
		\int_0^1\int_0^{L_g}\int_{Y^*_h}\mathcal{T}_\varepsilon^g(\mathcal{T}_\varepsilon(f(\varphi_\varepsilon)))\mathcal{T}_\varepsilon^g\mathcal{T}_\varepsilon\psi -f(\varphi) \psi \,dxdy_1dZ\\
		=
		\int_0^1\int_0^{L_g}\int_{Y^*_h}\mathcal{T}_\varepsilon^g(\mathcal{T}_\varepsilon(f(\varphi_\varepsilon)))\mathcal{T}_\varepsilon^g\mathcal{T}_\varepsilon\psi -f(\mathcal{T}_\varepsilon^g\mathcal{T}_\varepsilon\varphi) \mathcal{T}_\varepsilon^g\mathcal{T}_\varepsilon\psi \,dxdy_1dZ\\
		+\int_0^1\int_0^{L_g}\int_{Y^*_h}f(\mathcal{T}_\varepsilon^g\mathcal{T}_\varepsilon\varphi) \mathcal{T}_\varepsilon^g\mathcal{T}_\varepsilon\psi-f(\mathcal{T}_\varepsilon^g\mathcal{T}_\varepsilon\varphi) \psi  \,dxdy_1dZ\\
		+\int_0^1\int_0^{L_g}\int_{Y^*_h}f(\mathcal{T}_\varepsilon^g\mathcal{T}_\varepsilon\varphi)\psi-f(\varphi) \psi  \,dxdy_1dZ = I+II+III.
		\end{gathered}
		$$
		We show that each integral of the right hand side of the  above equality converges to zero.
		First, due to Propositions \ref{unfoldingproperties} and \ref{unfoldingpropertiesg}, we get 
		\begin{align*}
		|I|=&\left|\int_0^1\int_0^{L_g}\int_{Y^*_h}\mathcal{T}_\varepsilon^g(\mathcal{T}_\varepsilon(f(\varphi_\varepsilon)))\mathcal{T}_\varepsilon^g\mathcal{T}_\varepsilon\psi -f(\mathcal{T}_\varepsilon^g\mathcal{T}_\varepsilon\varphi) \mathcal{T}_\varepsilon^g\mathcal{T}_\varepsilon\psi \,dxdy_1dZ\right|\\& \leq \int_0^1\int_0^{L_g}\int_{Y^*_h}|\mathcal{T}_\varepsilon^g(\mathcal{T}_\varepsilon(f(\varphi_\varepsilon))) -f(\mathcal{T}_\varepsilon^g\mathcal{T}_\varepsilon\varphi) ||\mathcal{T}_\varepsilon^g\mathcal{T}_\varepsilon\psi| \,dxdy_1dZ \\& = \int_0^1\int_0^{L_g}\int_{Y^*_h}|f(\mathcal{T}_\varepsilon^g\mathcal{T}_\varepsilon\varphi_\varepsilon) -f(\mathcal{T}_\varepsilon^g\mathcal{T}_\varepsilon\varphi) ||\mathcal{T}_\varepsilon^g\mathcal{T}_\varepsilon\psi| \,dxdy_1dZ \\& \leq \int_0^1\int_0^{L_g}\int_{Y^*_h}\max_{x \in \mathbb{R}}|f'(x)||\mathcal{T}_\varepsilon^g\mathcal{T}_\varepsilon\varphi_\varepsilon-\mathcal{T}_\varepsilon^g\mathcal{T}_\varepsilon\varphi ||\mathcal{T}_\varepsilon^g\mathcal{T}_\varepsilon\psi| \,dxdy_1dZ \\ & \leq C \|\mathcal{T}_\varepsilon^g\mathcal{T}_\varepsilon\varphi_\varepsilon-\mathcal{T}_\varepsilon^g\mathcal{T}_\varepsilon\varphi\|_{L^p((0,1)\times(0,L_g)\times Y_h^*)}\|\mathcal{T}_\varepsilon^g\mathcal{T}_\varepsilon\psi\|_{L^{p'}((0,1)\times(0,L_g)\times Y_h^*)} \\& \leq K\left(||\varphi_\varepsilon-\varphi||^p_{L^p((0,1);W^{1,p}(Y^*))}-\dfrac{C}{\varepsilon^\gamma}\int_{\Lambda_\varepsilon^h\times Y^*_\varepsilon(x)}|\varphi_\varepsilon-\varphi|^p dxdY\right)  \to 0.
		\end{align*}

		On the other hand, by Propositions \ref{propconvtestfunction} and \ref{conv_unf}, since $f$ is bounded we have that
		\begin{align*}
		|II|=&\left|\int_0^1\int_0^{L_g}\int_{Y^*_h}f(\mathcal{T}_\varepsilon^g\mathcal{T}_\varepsilon\varphi) \mathcal{T}_\varepsilon^g\mathcal{T}_\varepsilon\psi-f(\mathcal{T}_\varepsilon^g\mathcal{T}_\varepsilon\varphi) \psi  \,dxdy_1dZ\right|\\& \leq \int_0^1\int_0^{L_g}\int_{Y^*_h}|f(\mathcal{T}_\varepsilon^g\mathcal{T}_\varepsilon\varphi)|| \mathcal{T}_\varepsilon^g\mathcal{T}_\varepsilon\psi-\psi| \,dxdy_1dZ \\ & \leq  \max_{x\in\mathbb{R}}|f(x)||Y^*_h|^{1/p'}(L_g)^{1/p'}\|\mathcal{T}_\varepsilon^g\mathcal{T}_\varepsilon\psi-\psi\|_{L^{p}((0,1)\times(0,L_g)\times Y_h^*)}\to 0.
		\end{align*}

		Finally, using again Propositions \ref{propconvtestfunction} and \ref{conv_unf} we get 
		\begin{align*}
		|III|=&\left|\int_0^1\int_0^{L_g}\int_{Y^*_h}f(\mathcal{T}_\varepsilon^g\mathcal{T}_\varepsilon\varphi)\psi-f(\varphi) \psi  \,dxdy_1dZ\right| \\& \leq \int_0^1\int_0^{L_g}\int_{Y^*_h}|f(\mathcal{T}_\varepsilon^g\mathcal{T}_\varepsilon\varphi)-f(\varphi)|\psi| \,dxdy_1dZ \\ & \leq  \max_{x\in\mathbb{R}}|f'(x)|\|\mathcal{T}_\varepsilon^g\mathcal{T}_\varepsilon\varphi-\varphi\|_{L^{p}((0,1)\times(0,L_g)\times Y_h^*)}L_g^{1/p'}|Y^*_h|^{1/p'}\|\psi\|_{L^{p'}(0,1)}\to 0
		\end{align*}
		as $\varepsilon\to0$ proving the result.
	
	\end{proof}

	\section{Convergence results} \label{linear}

	Consider the variational formulation \eqref{variational_l} from $p$-Laplacian equation \eqref{problem_l}. Our main goal here is to pass the limit on solutions $u_\varepsilon$ analyzing the different effects of the oscillatory boundary depending on the order $\alpha$.
	First we obtain the uniform boundedness to the solutions for any $\varepsilon>0$ and $\alpha>0$.

		\begin{proposition}\label{uniformbounds}
			Consider the variational formulation 
			\begin{equation} \label{1001}
			\begin{gathered}
			\int_{R^\varepsilon} \left\{ \left|\nabla u_\varepsilon\right|^{p-2}\nabla u_\varepsilon\nabla\varphi+\left|u_\varepsilon\right|^{p-2} u_\varepsilon\varphi \right\} dxdy =\dfrac{1}{\varepsilon^\gamma}\int_{\mathcal{O^\varepsilon}}f^\varepsilon\varphi dxdy, \; \varphi\in W^{1,p}(R^\varepsilon)
			\end{gathered}
			\end{equation}
			with $f^\varepsilon$ satisfying  
			$$\dfrac{1}{\varepsilon^{\gamma}}\left|\left|\left|f^\varepsilon\right|\right|\right|_{L^{p'}(\mathcal{O}^\varepsilon)}^{p'}\leq c$$
			for some positive constant $c$ independent of $\varepsilon>0$.
			
			Then, there exists $c_0>0$, also independent of $\varepsilon$ such that 
			$$
						\left|\left|\left|u_\varepsilon\right|\right|\right|_{W^{1,p}(R^\varepsilon)}\leq c_0.
			$$
			In particular
			\begin{equation*}
			\begin{gathered}
			\left|\left|\left|u_\varepsilon\right|\right|\right|_{L^{p}(R^\varepsilon)}\leq c_0 \quad  \textrm{ and } \quad 
			\left|\left|\left|\left|\nabla u_\varepsilon\right|^{p-2}\nabla u_\varepsilon\right|\right|\right|_{L^{p'}(R^\varepsilon)}\leq c_0^{p/p'}.
			\end{gathered}
			\end{equation*}
		\end{proposition}
		\begin{proof}
			Take $\varphi=u_\varepsilon$ in \eqref{1001}. Then, using Proposition \ref{bound}, 
			\begin{align*}
				\left|\left|u_\varepsilon\right|\right|_{W^{1,p}(R^\varepsilon)}^p  
				\leq \left(\dfrac{1}{\varepsilon^\gamma}\int_{\mathcal{O^\varepsilon}}|f^\varepsilon|^{p'}\right)^{1/p'} \left(\dfrac{1}{\varepsilon^\gamma}\int_{\mathcal{O^\varepsilon}}|u_\varepsilon|^{p}\right)^{1/p} 
				\leq c \, C \, |||u_\varepsilon|||_{W^{1,p}(R^\varepsilon)}.
			\end{align*}
			
			Hence, there exists $c_0>0$ such that 
			\begin{equation*}
			\left|\left|\left|u_\varepsilon\right|\right|\right|_{W^{1,p}(R^\varepsilon)}\leq c.
			\end{equation*}
			
			Therefore, $u_\varepsilon$ and $|\nabla u_\varepsilon|^{p-2}\nabla u_\varepsilon$ are respectively uniformly bounded in $L^p(R^\varepsilon)$ and $(L^{p'}(R^\varepsilon))^2$ under norm $\left|\left|\left|\cdot\right|\right|\right|$.
		\end{proof}
	
	Now, let us perform the asymptotic analysis for each $\alpha>0$.

	\subsection{The linear perturbation case}
	
	Let us now deal with the convergence of the problem \eqref{variational_l} at \eqref{TDs} and its dependence on $\alpha$. Combining convergence results from Section \ref{pre} and \cite{nakasato1,nakasato}, we obtain the following theorem.
	
	\begin{theorem}  \label{alpha1}
		Let $u_\varepsilon\in W^{1,p}(R^{\varepsilon})$ family of solutions given by \eqref{variational_l} at the thin domain \eqref{TDs}. Assume $f^\varepsilon\in L^{p'}(R^\varepsilon)$ satisfying 
		$$
		1/ \varepsilon^{\gamma} |||f^\varepsilon|||^{p'}_{L^{p'}(\mathcal{O}^\varepsilon)}\leq C
		$$ 
		for some $C>0$ independent of $\varepsilon>0$ and 
		$$
		\mathcal{T}_\varepsilon^g(\mathcal{T}_\varepsilon f^\varepsilon)\rightarrow \hat{f} \textrm{ weakly in } L^{p'}((0,1)\times(0,L_g)\times Y_h^*))
		$$ 
		for some $\hat{f}\in L^{p'}(0,1)$.
		
		Then there exists $u\in W^{1,p}(0,1)$ such that 
		$$\mathcal{T}_\varepsilon u_\varepsilon \to u \textrm{ strongly in } L^p(0,1;W^{1,p}(Y^*)),$$
		where $u$ is the unique solution of the homogenized equation \eqref{limite_s} satisfying
		\begin{equation}\label{limit_a00}
		\int_0^1q|\nabla u|^{p-2}\nabla u\nabla \varphi + |u|^{p-2} u \varphi = \int_0^1\bar{f}\varphi, \qquad \forall \varphi\in W^{1,p}(0,1)
		\end{equation}
		 with
		$$\bar{f}=\dfrac{1}{L_h|Y^{*}|}\int_{(0,L_g)\times Y_h^*}\hat{f}dy_1dZ$$
		such that:		
		\begin{enumerate}[(i)]
			\item if $\alpha = 1$ in \eqref{TDs}, we have
	$$q=\dfrac{1}{|Y^*|}\int_{Y^*}|\nabla v|^{p-2}\partial_{y_1}vdy_1dy_2$$
	where $Y^*$ is the representative cell of $R^\varepsilon$ and $v$ is the auxiliary function given by
	\begin{align}\label{thm3.2auxiliarproblem}
	\int_{Y^*}|\nabla v|^{p-2}\nabla v \nabla \varphi dy_1dy_2=0 \ \forall\varphi\in W_{\#}^{1,p}(Y^*) 
	\\(v-y_1)\in W_{\#}^{1,p}(Y^*) \text{ with } \langle(v-y_1)\rangle_{Y^*}=0. \nonumber
	\end{align}
	
	\item if $\alpha < 1$ in \eqref{TDs}, we have
	$$q=\dfrac{1}{\langle g\rangle_{(0,L_{g})}\langle 1/g^{p'-1}\rangle_{(0,L_{g})}^{p-1}}.$$
	\item if $\alpha > 1$ in \eqref{TDs}, we have
	$$q=\dfrac{g_{0}}{\langle g\rangle_{(0,L_{g})}}.$$
\end{enumerate}
	\end{theorem}
	\begin{proof}	
		From Propositions \ref{unfoldingproperties}, \ref{iteration} and \ref{unfoldingpropertiesg}, we can rewrite \eqref{variational} as
		\begin{align}
		\label{var=1}
		\nonumber\int_{(0,1)\times Y^*}&\mathcal{T}_\varepsilon(|\nabla u_\varepsilon|^{p-2}\nabla u_\varepsilon)\mathcal{T}_\varepsilon\nabla\varphi dxdy_1dy_2+\dfrac{L_g}{\varepsilon}\int_{R_1^\varepsilon}|\nabla u_\varepsilon|^{p-2}\nabla u_\varepsilon\nabla \varphi dxdy \ + \\ & \nonumber + \int_{(0,1)\times Y^*}\mathcal{T}_\varepsilon(|u_\varepsilon|^{p-2} u_\varepsilon)\mathcal{T}_\varepsilon\varphi dxdy_1dy_2+\dfrac{L_g}{\varepsilon}\int_{R_1^\varepsilon}|u_\varepsilon|^{p-2} u_\varepsilon \varphi dxdy  \\  = \ & \dfrac{1}{L_h}\int_0^1\int_0^{L_g}\int_{Y_h^*}\mathcal{T}_\varepsilon^g(\mathcal{T}_\varepsilon f^\varepsilon)\mathcal{T}_\varepsilon^g(\mathcal{T}_\varepsilon \varphi)dz_1dz_2dy_1dx\ +\nonumber \\ & +\dfrac{1}{\varepsilon^\gamma}\int_{\Lambda_\varepsilon^h\times Y^*_\varepsilon(x)}\mathcal{T}_\varepsilon f^\varepsilon \mathcal{T}_\varepsilon\varphi dy_1dy_2dx+\dfrac{L_g}{\varepsilon^{\gamma+1}}\int_{\mathcal{O}^\varepsilon_1}f^\varepsilon\varphi dxdy
		\end{align}
		
		Using property \ref{uci} from Proposition \ref{unfoldingproperties}, we have
		$$\dfrac{L_g}{\varepsilon}\int_{R_1^\varepsilon}|\nabla u_\varepsilon|^{p-2}\nabla u_\varepsilon\nabla \varphi dxdy\to 0
		\ \text{ and } \
		\dfrac{L_g}{\varepsilon}\int_{R_1^\varepsilon}|u_\varepsilon|^{p-2} u_\varepsilon \varphi dxdy\to 0$$
		when $\varepsilon\to0$.
		
		Consequently, it follows from \cite[Theorem 3.1]{nakasato1} then there exists $u\in W^{1,p}(0,1)$ such that $\mathcal{T}_\varepsilon u_\varepsilon \to u$ strongly in $L^p(0,1;W^{1,p}(Y^*))$ and
		\begin{align}
		\int_{(0,1)\times Y^*}&\mathcal{T}_\varepsilon(|\nabla u_\varepsilon|^{p-2}\nabla u_\varepsilon)\mathcal{T}_\varepsilon\nabla\varphi dxdy_1dy_2+\dfrac{L_g}{\varepsilon}\int_{R_1^\varepsilon}|\nabla u_\varepsilon|^{p-2}\nabla u_\varepsilon\nabla \varphi dxdy \ + \nonumber \\ & + \int_{(0,1)\times Y^*}\mathcal{T}_\varepsilon(|u_\varepsilon|^{p-2} u_\varepsilon)\mathcal{T}_\varepsilon\varphi dxdy_1dy_2+\dfrac{L_g}{\varepsilon}\int_{R_1^\varepsilon}|u_\varepsilon|^{p-2} u_\varepsilon \varphi dxdy \nonumber \\ & \to |Y^{*}|\int_0^1(q|\partial_xu|^{p-2}\partial_xu\partial_x\varphi+|u|^{p-2}u\varphi )dx, \forall \varphi\in W^{1,p}(0,1),
		\end{align}
		with 
		$$q=\dfrac{1}{|Y^*|}\int_{Y^*}|\nabla v|^{p-2}\partial_{y_1}vdy_1dy_2$$
		where $Y^*$ is the representative cell and $v$ is an auxiliary function given by \eqref{thm3.2auxiliarproblem}.		
		%
		%
		
		Analogously, we can use Proposition \ref{uciO1} to get 
		$$\dfrac{L_g}{\varepsilon^{\gamma+1}}\int_{\mathcal{O}^\varepsilon_1}f^\varepsilon\varphi dxdy\to 0,$$
		and Proposition \ref{propboundnessuci} item (b) to conclude 
		$$\dfrac{1}{\varepsilon^\gamma}\int_{\Lambda_\varepsilon^h\times Y^*_\varepsilon(x)}\mathcal{T}_\varepsilon f^\varepsilon \mathcal{T}_\varepsilon\varphi dy_1dy_2dx\to 0$$
		as $\varepsilon\to0$.
		
		Hence, it remains analyzing the behavior of 
		$$\dfrac{1}{L_h}\int_0^1\int_0^{L_g}\int_{Y_h^*}\mathcal{T}_\varepsilon^g(\mathcal{T}_\varepsilon f^\varepsilon)\mathcal{T}_\varepsilon^g(\mathcal{T}_\varepsilon \varphi)dz_1dz_2dy_1dx$$
		to conclude the proof. Since, by hypothesis, we have that there exists	$\hat{f}\in L^{p'}(0,1)$ such that $\mathcal{T}_\varepsilon^g(\mathcal{T}_\varepsilon f^\varepsilon)\rightarrow \hat{f}$ weakly in $L^{p'}((0,1)\times(0,L_g)\times Y_h^*))$, we obtain		
		by Corollary \ref{conv_lin} that there exists $\bar{f}\in L^{p'}(0,1)$ such that
		\begin{align*}
		\dfrac{1}{L_h}&\int_0^1\int_0^{L_g}\int_{Y_h^*}\mathcal{T}_\varepsilon^g(\mathcal{T}_\varepsilon f^\varepsilon)\mathcal{T}_\varepsilon^g(\mathcal{T}_\varepsilon \varphi)dZdy_1dx \to \int_0^1\bar{f} \varphi  dx
		\end{align*}
		with $$\bar{f}=\dfrac{1}{L_h|Y^{*}|}\int_{(0,L_g)\times Y_h^*}\hat{f}dy_1dZ.$$
		
		Now, passing to the limit in \eqref{var=1}, we get
		\begin{align*}
		|Y^{*}|\int_0^1(q|\partial_xu|^{p-2}\partial_xu\partial_x\varphi+|u|^{p-2}u\varphi )dx=|Y^{*}|\int_0^1 \bar{f} \varphi  dx, \forall \varphi\in W^{1,p}(0,1),
		\end{align*}
		which implies \eqref{limit_a00} leading us to the end of the proof.
	\end{proof}
	\section{A nonlinear perturbation}  \label{nonlinear}

	Now, let us consider the variational formulation of \eqref{problem_s}, with $p\geq 2$, under the condition ($\mathbf{H_f}$): 
	\begin{equation}\label{variational}\int_{R^\varepsilon} \{|\nabla u_\varepsilon|^{p-2}\nabla u_\varepsilon\nabla \varphi+|u_\varepsilon|^{p-2}u_\varepsilon\varphi \}dxdy=\dfrac{1}{\varepsilon^\gamma}\int_{\mathcal{O}^\varepsilon}f(u_\varepsilon)\varphi dxdy, 
	\quad \forall \varphi \in W^{1,p}(R^\varepsilon). \end{equation}
	Here, we see that the homogenized equation for the thin domain problem \eqref{variational} is the one-dimensional equation \eqref{limite_s}  given by
	\begin{equation}\label{variational_lim}\int_{0}^{1} \{q|u'|^{p-2}u' \varphi'+|u|^{p-2}u\varphi \}dx =\dfrac{\langle h\rangle_{(0,L_h)}}{\langle g\rangle_{(0,L_g)}}\int_0^1 f(u)\varphi dx, 
\quad \forall \varphi \in W^{1,p}(0,1). \end{equation}

	\subsection{Existence of solutions}
	
	Using the notation of Proposition \ref{propositionduallitymap}, we can rewrite problem \eqref{variational} in the following way
	$$
	J u_\varepsilon =F_\varepsilon(u_\varepsilon),
	$$
	where, for $1-1/p<s<1$, 
	\begin{align}\label{def_f}
	F_\varepsilon : W^{1,p}(R^\varepsilon)&\to (L^p((0,1);W^{s,p}(0,\varepsilon g(x/\varepsilon^\alpha)))' \nonumber \\
	u_\varepsilon & \mapsto F_\varepsilon(u_\varepsilon) : L^p((0,1);W^{s,p}(0,\varepsilon g(x/\varepsilon^\alpha))) \to \mathbb{R} \\ & \qquad \qquad \qquad \qquad  \varphi_\varepsilon \mapsto \langle F_\varepsilon(u_\varepsilon),\varphi_\varepsilon\rangle=\dfrac{1}{\varepsilon^\gamma}\int_{\mathcal{O}^\varepsilon}f(u_\varepsilon)\varphi_\varepsilon. \nonumber 
	\end{align}
%
	Arguing as \cite[Proposition 6]{Nogueira} and \cite[Proposition 4.1]{AMA-thin}, one can prove:
	\begin{proposition}\label{f}
	If we call the spaces $$Z_\varepsilon=L^p((0,1);W^{s,p}(0,\varepsilon g(x/\varepsilon^\alpha))) \text{ and } Z_\varepsilon'= (L^{p}((0,1);W^{s,p}(0,\varepsilon g(x/\varepsilon^\alpha))))',$$ the function $F_\varepsilon$ defined in \eqref{def_f} has the following properties:
	\begin{enumerate}
	    \item there exists $K>0$ independent of $\varepsilon$ such that $$\sup_{u_\varepsilon\in W^{1,p}(R^\varepsilon)}\|F_\varepsilon(u_\varepsilon)\|_{Z_\varepsilon'}\leq K$$
	    \item $F_\varepsilon$ is a Lipschitz application and there exists $L>0$ independent of $\varepsilon$ such that
	    $$\|F_\varepsilon(u_\varepsilon)-F_\varepsilon(v_\varepsilon)\|_{Z_\varepsilon'}\leq L \|u_\varepsilon-v_\varepsilon\|_{W^{1,p}(R^\varepsilon)}$$
	\end{enumerate}
	\end{proposition}
	\begin{proof}
	\begin{enumerate}
	    \item For $u_\varepsilon\in Z_\varepsilon$, $$\|F_\varepsilon(u_\varepsilon)\|_{Z_\varepsilon'}=\sup_{ \|\varphi_\varepsilon\|_{Z_\varepsilon}=1}|\langle F_\varepsilon(u_\varepsilon),\varphi_\varepsilon \rangle|.$$
	    
	    Then, using that $f$ is bounded and by Proposition \ref{bound}, if $q$ is the conjugate of $p$,
	    \begin{align*}
	        \dfrac{1}{\varepsilon^\gamma}\int_{\mathcal{O}^\varepsilon}&|f(u_\varepsilon)\varphi_\varepsilon| \leq \left(\dfrac{1}{\varepsilon^{\gamma}}\int_{\mathcal{O}^\varepsilon}|f(u_\varepsilon)|^q\right)^{1/q}\left(\dfrac{1}{\varepsilon^{\gamma}}\int_{\mathcal{O}^\varepsilon}|\varphi_\varepsilon|^p\right)^{1/p} \\ & \leq C\sup_{x\in\mathbb{R}}|f(x)|h_1^{1/q}\|\varphi_\varepsilon\|_{Z_\varepsilon}
	    \end{align*}
	    that implies 
	    $$\|F_\varepsilon(u_\varepsilon)\|_{Z_\varepsilon'}\leq \sup_{x\in\mathbb{R}}|f(x)|h_1^{1/q}.$$
	    Thus there exists $K>0$ such that 
	    $$\sup_{u_\varepsilon\in W^{1,p}(R^\varepsilon)}\|F_\varepsilon(u_\varepsilon)\|_{Z_\varepsilon'}\leq K$$
	    
	    \item Using now $f'$ bounded, the Mean Value Theorem and Proposition \ref{bound}, if $q$ is the conjugate of $p$,
	    \begin{align*}\dfrac{1}{\varepsilon^\gamma}\int_{\mathcal{O}^\varepsilon}&|f(u_\varepsilon)\varphi_\varepsilon-f(v_\varepsilon)\varphi_\varepsilon|  \leq \dfrac{1}{\varepsilon^{\gamma}}\int_{\mathcal{O}^\varepsilon}|f(u_\varepsilon)-f(v_\varepsilon)||\varphi_\varepsilon| \\ &  \leq \left(\dfrac{1}{\varepsilon^{\gamma}}\int_{\mathcal{O}^\varepsilon}\sup_{x\in\mathbb{R}}|f'(x)|^q|u_\varepsilon - v_\varepsilon|^q\right)^{1/q}\left(\dfrac{1}{\varepsilon^{\gamma}}\int_{\mathcal{O}^\varepsilon}|\varphi_\varepsilon|^p\right)^{1/p} \\ & \leq C\sup_{x\in\mathbb{R}}|f'(x)| \,  \|u_\varepsilon-v_\varepsilon\|_{W^{1,p}(R^\varepsilon)} \|\varphi_\varepsilon\|_{Z_\varepsilon}
	    \end{align*}
	    Therefore for some $L>0$ we obtain
	    $$\|F_\varepsilon(u_\varepsilon)-F_\varepsilon(v_\varepsilon)\|_{Z_\varepsilon'}\leq L \|u_\varepsilon-v_\varepsilon\|_{W^{1,p}(R^\varepsilon)}.$$
	   \end{enumerate}
	\end{proof}
	
	By Proposition \ref{propositionduallitymap}, $J$ is  bijective, with inverse bounded, continuous and monotone. Using Proposition \ref{f}, the application $F_\varepsilon$ is Lipschitz continuous with constant independent of $\varepsilon$. Thus, $u_\varepsilon\in W^{1,p}(R^\varepsilon)$ is a solution of \eqref{variational} if, and only if, $u_\varepsilon=J^{-1}F_\varepsilon(u_\varepsilon)$. This means that $u_\varepsilon$ is a fixed point of $J^{-1} F_\varepsilon$. Also, one gets that $J^{-1} F_\varepsilon$ is a compact and continuous operator since $W^{1,p}(R^\varepsilon)\subset L^p((0,1);W^{s,p}(0,\varepsilon g(x/\varepsilon^\alpha)))$ is a compact inclusion. Thus, by the Schaefer's Fixed Point Theorem, we get the existence of solutions for each $\varepsilon>0$.
	Analogously, for our limiting problem the existence of solutions follows considering \eqref{variational_lim} as
	$$
	J u =F_0(u),
	$$
	where
	\begin{align*}
	F_0 : W^{1,p}(0,1)&\to (L^p(0,1))'  \\
	u & \mapsto F_0(u) : L^{p'}(0,1) \to \mathbb{R} \\ & \qquad \qquad \qquad  \varphi \mapsto \langle F_0(u),\varphi\rangle=\dfrac{\langle h\rangle_{(0,L_h)}}{\langle g\rangle_{(0,L_g)}}\int_0^1f(u)\varphi.  
	\end{align*}

	\subsection{Convergence of solutions}
		
	Our main goal now is to pass the limit, as $\varepsilon\to0$, in equation \eqref{variational} analyzing the different effects of the oscillatory boundary depending on order $\alpha$. Using convergence results from \cite{nakasato1,nakasato} we have:
		\begin{theorem}\label{conv_nl}
		Let $u_\varepsilon\in W^{1,p}(R)$ solution of the problem \eqref{variational}. 	
		Then, there exists $u\in W^{1,p}(0,1)$, such that 
			$$\mathcal{T}_\varepsilon u_\varepsilon \to u \textrm{ strongly in } L^p(0,1;W^{1,p}(Y^*)),$$
         where $u$ is the unique solution of the problem \eqref{variational_lim} such that  
		\begin{enumerate}[(i)]
			\item if $\alpha = 1$ in \eqref{TDs}, we have
	$$q=\dfrac{1}{|Y^*|}\int_{Y^*}|\nabla v|^{p-2}\partial_{y_1}vdy_1dy_2$$
	where $Y^*$ is the representative cell of $R^\varepsilon$ and $v$ is the auxiliary function given by
	\begin{align}\label{thm4.1auxiliarproblem}
	\int_{Y^*}|\nabla v|^{p-2}\nabla v \nabla \varphi dy_1dy_2=0 \ \forall\varphi\in W_{\#}^{1,p}(Y^*) 
	\\	(v-y_1)\in W_{\#}^{1,p}(Y^*) \text{ with } \langle(v-y_1)\rangle_{Y^*}=0. \nonumber
	\end{align}

	\item If $\alpha < 1$ in \eqref{TDs}, we have
	$$q=\dfrac{1}{\langle g\rangle_{(0,L_{g})}\langle 1/g^{p'-1}\rangle_{(0,L_{g})}^{p-1}}.$$
	\item If $\alpha > 1$ in \eqref{TDs}, we have
	$$q=\dfrac{g_{0}}{\langle g\rangle_{(0,L_{g})}}.$$
\end{enumerate}
		\end{theorem}
	
	\begin{proof}
		Taking $\alpha = 1$, from Propositions \ref{unfoldingproperties}, \ref{unfoldingpropertiesg} and \ref{iteration} we can rewrite \eqref{variational} as
		\begin{align*}
		\label{var=101}
		\int_{(0,1)\times Y^*}&\mathcal{T}_\varepsilon(|\nabla u_\varepsilon|^{p-2}\nabla u_\varepsilon)\mathcal{T}_\varepsilon\nabla\varphi dxdy_1dy_2+\dfrac{L_g}{\varepsilon}\int_{R_1^\varepsilon}|\nabla u_\varepsilon|^{p-2}\nabla u_\varepsilon\nabla \varphi dxdy \ + \\ & + \int_{(0,1)\times Y^*}\mathcal{T}_\varepsilon(|u_\varepsilon|^{p-2} u_\varepsilon)\mathcal{T}_\varepsilon\varphi dxdy_1dy_2+\dfrac{L_g}{\varepsilon}\int_{R_1^\varepsilon}|u_\varepsilon|^{p-2} u_\varepsilon \varphi dxdy \ = \\ & = \dfrac{1}{L_h}\int_0^1\int_0^{L_g}\int_{Y_h^*}\mathcal{T}_\varepsilon^g(\mathcal{T}_\varepsilon f(u_\varepsilon))\mathcal{T}_\varepsilon^g(\mathcal{T}_\varepsilon \varphi)dz_1dz_2dy_1dx\ + \\ & +\dfrac{1}{\varepsilon^\gamma}\int_{\Lambda_\varepsilon^h\times Y^*_\varepsilon(x)}\mathcal{T}_\varepsilon(f(u_\varepsilon))\mathcal{T}_\varepsilon\varphi dy_1dy_2dx+\dfrac{L_g}{\varepsilon^{\gamma+1}}\int_{\mathcal{O}^\varepsilon_1}f(u^\varepsilon)\varphi dxdy
		\end{align*}
		
		Using property \ref{uci} from Proposition \ref{unfoldingproperties}, we have
		$$\dfrac{L_g}{\varepsilon}\int_{R_1^\varepsilon}|\nabla u_\varepsilon|^{p-2}\nabla u_\varepsilon\nabla \varphi dxdy\to 0
		\ \text{ and } \
		\dfrac{L_g}{\varepsilon}\int_{R_1^\varepsilon}|u_\varepsilon|^{p-2} u_\varepsilon \varphi dxdy\to 0$$
		when $\varepsilon\to0$.
		
		Hence, one can apply similar arguments from \cite[Theorem 3.1]{nakasato1} showing the existence of $u\in W^{1,p}(0,1)$ such that $\mathcal{T}_\varepsilon u_\varepsilon \to u$ strongly in $L^p(0,1;W^{1,p}(Y^*))$ and
		\begin{align}
		\int_{(0,1)\times Y^*}&\mathcal{T}_\varepsilon(|\nabla u_\varepsilon|^{p-2}\nabla u_\varepsilon)\mathcal{T}_\varepsilon\nabla\varphi dxdy_1dy_2+\dfrac{L_g}{\varepsilon}\int_{R_1^\varepsilon}|\nabla u_\varepsilon|^{p-2}\nabla u_\varepsilon\nabla \varphi dxdy \ + \nonumber \\ & + \int_{(0,1)\times Y^*}\mathcal{T}_\varepsilon(|u_\varepsilon|^{p-2} u_\varepsilon)\mathcal{T}_\varepsilon\varphi dxdy_1dy_2+\dfrac{L_g}{\varepsilon}\int_{R_1^\varepsilon}|u_\varepsilon|^{p-2} u_\varepsilon \varphi dxdy \nonumber \\ & \to |Y^{*}|\int_0^1(q|\partial_xu|^{p-2}\partial_xu\partial_x\varphi+|u|^{p-2}u\varphi )dx, \forall \varphi\in W^{1,p}(0,1), \nonumber
		\end{align}
		with 
		$$q=\dfrac{1}{|Y^*|}\int_{Y^*}|\nabla v|^{p-2}\partial_{y_1}vdy_1dy_2$$
		where $Y^*$ is the representative cell of $R^\varepsilon$ and $v$ is an auxiliary function given by \eqref{thm4.1auxiliarproblem}.		
%
%
		
		Analogously, using respectively Propositions \ref{uciO1} and \ref{propboundnessuci} (b), one gets 
		$$\dfrac{L_g}{\varepsilon^{\gamma+1}}\int_{\mathcal{O}^\varepsilon_1}f(u^\varepsilon)\varphi dxdy\to 0
		\quad \textrm{and} \quad  
		\dfrac{1}{\varepsilon^\gamma}\int_{\Lambda_\varepsilon^h\times Y^*_\varepsilon(x)}\mathcal{T}_\varepsilon(f(u_\varepsilon))\mathcal{T}_\varepsilon\varphi dy_1dy_2dx\to 0$$
		as $\varepsilon\to0$.
		Hence, it remains to analyze the behavior of 
		$$\dfrac{1}{L_h}\int_0^1\int_0^{L_g}\int_{Y_h^*}\mathcal{T}_\varepsilon^g(\mathcal{T}_\varepsilon f(u_\varepsilon))\mathcal{T}_\varepsilon^g(\mathcal{T}_\varepsilon \varphi)dz_1dz_2dy_1dx$$
		to conclude the proof. 
		
		By Proposition \ref{conv_f}, 
		since $\mathcal{T}_\varepsilon u_\varepsilon \to u$ strongly in $L^p((0,1);W^{1,p}(Y^*))$, for $\varphi\in W^{1,p}(0,1)$, previous integral converges to 
		\begin{align*}
		\dfrac{1}{L_{h}}\int_0^1\int_0^{L_g}\int_{Y^*_h}f(u) \varphi dz_1dz_2dy_1dx.
		\end{align*}
			
		Lastly, notice that 
		\begin{align*}
		\dfrac{1}{L_{h}}&\int_0^1\int_0^{L_g}\int_{Y^*_h}f(u(x)) \varphi(x) dz_1dz_2dy_1dx  = \dfrac{L_g}{L_{h}}\int_0^{L_h}\int_{-h(z_1)}^0\int_{0}^{1}f(u(x)) \varphi(x) dz_2dz_1dx \\&= \dfrac{L_g}{L_{h}}\int_0^{L_h}\int_{0}^{1}h(z_1)f(u(x)) \varphi(x) dz_1dx=L_g\langle h\rangle_{(0,L_{h})}\int_0^1f(u(x)) \varphi(x) dx 
		\end{align*}
	
		This leads us to the following limit problem
		\begin{equation*}
		q\int_{0}^{1}|\partial_{x}u|^{p-2}\partial_{x}u\;\partial_{x}\varphi \; dx +\int_{0}^{1}|u|^{p-2}u \varphi dx=\dfrac{\langle g\rangle_{(0,L_{g})}}{\langle h\rangle_{(0,L_{h})}}\int_{0}^{1}f(u) \varphi dx
		\end{equation*}
		that concludes our proof. The cases $\alpha<1$ and $\alpha>1$ follow similarly using arguments from \cite[Theorem 4.1, Theorem 5.3]{nakasato1} respectively.

	\end{proof}	

\begin{remark}
	If we analyze the problem \eqref{variational} with $p=2$ and $\alpha=1$, the Theorem \ref{conv_nl} is equivalent to \cite[Proposition 5.10]{AMA-thin}.
\end{remark}

\end{document}